\documentclass[oneside,english,ngerman,british]{amsart}
\usepackage{amsaddr}
\usepackage{palatino}
\usepackage[T1]{fontenc}
\usepackage[latin9]{inputenc}
\usepackage{textcomp}
\usepackage{mathtools}
\usepackage{amstext}
\usepackage{amsthm}
\usepackage{amssymb}
\usepackage{stackrel}
\usepackage{graphicx}
\usepackage{caption}
\usepackage{setspace}
\usepackage{wasysym}
\usepackage{enumitem}
\usepackage[usenames]{color}
\usepackage{colortbl,todonotes}
\usepackage{amsmath} 
{
      \theoremstyle{plain}
      \newtheorem{assumption}{Assumption}
  }

\makeatletter
\numberwithin{equation}{section}
\numberwithin{figure}{section}
\theoremstyle{plain}
\newtheorem{thm}{\protect\theoremname}
\theoremstyle{plain}

\theoremstyle{plain}
\newtheorem{lem}[thm]{\protect\lemmaname}
\graphicspath{{pictures/}}
\DeclareGraphicsExtensions{.pdf,.png,.jpg}
\makeatother

\usepackage{babel}
\addto\captionsbritish{\renewcommand{\lemmaname}{Lemma}}
\addto\captionsbritish{\renewcommand{\propositionname}{Proposition}}
\addto\captionsbritish{\renewcommand{\theoremname}{Theorem}}
\addto\captionsenglish{\renewcommand{\lemmaname}{Lemma}}
\addto\captionsenglish{\renewcommand{\propositionname}{Proposition}}
\addto\captionsenglish{\renewcommand{\theoremname}{Theorem}}
\addto\captionsngerman{\renewcommand{\lemmaname}{Lemma}}
\addto\captionsngerman{\renewcommand{\propositionname}{Satz}}
\addto\captionsngerman{\renewcommand{\theoremname}{Theorem}}
\providecommand{\lemmaname}{Lemma}
\providecommand{\propositionname}{Proposition}
\providecommand{\theoremname}{Theorem}

\begin{document}

\title{Spectral bootstrap confidence bands for L\'evy-driven
moving average processes\footnote{The support of German Science Foundation research grant (DFG Sachbeihilfe) 406700014 is gratefully acknowledged.
}}
\author{D. Belomestny, E. Ivanova and T. Orlova}
\address{Duisburg-Essen University}
\begin{abstract}
In this paper we study the problem of constructing bootstrap confidence
intervals for the L\'evy density of the driving
L\'evy  process based on high-frequency observations
of a L\'evy-driven moving average processes.
Using a spectral estimator of the L\'evy 
density, we propose a novel implementations of multiplier and empirical
bootstraps to construct confidence bands on a compact set away from
the origin. We also provide conditions under which the confidence
bands are asymptotically valid. 
\end{abstract}

\maketitle

\section{Introduction}

The continuous-time L\'evy-driven moving average processes
are defined as 
\begin{eqnarray}
Z_{t}=\int_{-\infty}^{\infty}\mathcal{K}(t-s)\, dL_{s}\label{Zt}
\end{eqnarray}
where $\mathcal{K}$ is a deterministic kernel and $L=\left(L_{t}\right)_{t\in\mathbb{R}}$
is a two-sided L\'evy process with a L\'evy triplet $(\gamma,\sigma,\nu)$.
The conditions which guarantee that this integral is well-defined
are given in the pioneering work by Rajput and Rosinski \cite{rajput1989spectral}. 
For instance, if $\int x^{2}\nu\left(dx\right)<\infty$, it is sufficient to
assume that $\mathcal{K}\in\mathcal{L}^{1}\left(\mathbb{R}\right)\cap\mathcal{L}^{2}(\mathbb{R})$.
\par
Continuous-time L\'{e}vy-driven moving average processes (and slightly modified versions of them) are widely used for the construction of many popular models such as L\'{e}vy-driven Ornstein-Uhlenbeck processes, fractional L\'{e}vy processes, CARMA processes, L\'{e}vy semistationary processes and ambit fields, cf. Barndorff-Nielsen, Benth and Veraart \cite{Barndorff-NielsenBenthVeraart2015}, Podolskij  \cite{Pod}.  Most of these models  can be applied to financial and physical problems. For instance, the choice  $\mathcal{K}(t)=t^{\alpha}e^{-\lambda t}1_{[0,\infty)}(t)$
with $\lambda>0$ and $\alpha>-1/2$ (known as Gamma-kernel) is used for modeling volatility and turbulence, see e.g. Barndorff-Nielsen and Schmiegel \cite{Barndorff-NielsenSchmiegel2009}. Otherwise, the choice $\mathcal{K}(t)=e^{-\lambda|t|}$ (known as well-balanced Ornstein-Uhlenbeck process) can be used for the analysis of the SAP high-frequency data, see Schnurr and Woerner~\cite{WD}.
\par
This paper is devoted to  statistical inference for  continuous-time L\'evy-driven moving average processes. Assuming that the high-frequency equidistant observations of the process \((Z_t)\) are given, we aim to estimate the characteristic triplet of the process \((L_t)\). Recently, Belomestny, Panov and Woerner \cite{BPW} considered the
statistical estimation of the L\'evy measure $\nu$ from the low-frequency
observations of the process $(Z_{t})$. The approach presented in
\cite{BPW} is rather general - in particular, it works well under
various choices of $\mathcal{K}$. Nevertheless, this approach is
based on the superposition of the Mellin and Fourier transforms of
the L\'evy measure, and therefore its practical implementation can meet
some computational difficulties. 
In \cite{belomestny2019statistical}, another method was presented, which essentially uses the following theoretical observation.  For any kernel \(\mathcal{K}\), the characteristic function \(\Phi(u):=\mathsf{E}\left[e^{\mathrm{i}uZ_{t}}\right]\) of the process \((Z_t)\) and the characteristic exponent \(\psi(u) \) of the process \((L_t)\) are connected via the formula: 
\begin{eqnarray*}
\Phi(u)=\exp\left (
\int_{\mathbb{R}}\psi(u\,\mathcal{K}(s))\, ds
\right).
\end{eqnarray*}
It was noted in \cite{belomestny2019statistical} that under the choice  \(\mathcal{K}(x)=\left(1-\alpha|x|\right)_{+}^{\frac{1}{\alpha}}\) this formula can be inverted without use of an additional integral transformations, that is, the function \(\psi\) can be represented via \(\Phi\) and its derivatives. Therefore, the characteristic exponent
can be estimated from the observations of the process $(Z_t),$ and further
application of the Fourier techniques leads to 
a consistent estimator of the L\'evy triplet.  
\par
The current  paper is devoted to the estimation of the L\'evy measure \(\nu\) in the same model as in \cite{belomestny2019statistical} but based on high-frequency observations of the process \((Z_t).\) Moreover, we are interested in  uniform bootstrap confidence bands for \(\nu. \) We propose a novel implementation of the multiplier and empirical
bootstrap procedures to construct confidence bands on a compact set away from
the origin. We also provide conditions under which the confidence
bands are asymptotically valid. Our approach can be viewed as an extension of the recent work \cite{kato2017bootstrap} where bootstrap confidence bands are constructed for the case of high-frequency observations of the L\'evy process \((L_t)_{t\geq 0}\) itself. 
\par
The paper is organised as follows. In Section~\ref{sec:setup} we formulate   our main statistical problem and propose an estimator for the underlying L\'evy density \(\nu\). We also discuss how to construct confidence bands for \(\nu\). Section~\ref{sec:conf} contains a detailed description of the bootstrap procedure and results on the validity of the bootstrap confidence bands. Some numerical results on simulated data are shown in Section~\ref{sec:num}. Finally, in Section~\ref{sec:proofs} all proofs are collected. 

\section{Set-up}
\label{sec:setup}
We shall consider continuous-time L\'evy-driven
moving average processes $(Z_{t})_{t\geq 0}$  of the form:
\begin{eqnarray}
Z_{t}=\int_{-\infty}^{\infty}\mathcal{K}(t-s)\,dL_{s},\label{Zt}
\end{eqnarray}
where $\mathcal{K}$ is a symmetric kernel given by
\begin{eqnarray}
\mathcal{K}_{\alpha}(x)=
\begin{cases}
\left(1-\alpha|x|\right)^{\frac{1}{\alpha}}, & \quad|x|\leq\alpha^{-1}, 
\\
0, & \mathrm{else}
\end{cases}
\label{Kalpha}
\end{eqnarray}
for some $\alpha\in(0,1),$  $L=\left(L_{t}\right)_{t\in\mathbb{R}}$ is a
two-sided L\'evy process with the L\'evy triplet $\left(\gamma,\sigma,\nu\right)$. Note that as a limiting case for $\alpha\to0$ we get the exponential kernel $\mathcal{K}_{0}(x)=\exp\left(-x\right)$. It follows from \cite{rajput1989spectral} that the process \((Z_t)\) is well-defined and infinitely divisible with the characteristic function: 
\begin{eqnarray*}
\mathsf{E}\left[e^{iuZ_{t}}\right]=\exp\left\{ \mathrm{i} u\gamma_{Z}(t)-\frac{1}{2}u^{2}\sigma_{Z}^{2}(t)+\int_{\mathbb{R}}\left[e^{\mathrm{i} ux}-1-iux\mathrm{1}_{\{|x|\leq1\}}\right]\,\nu_{Z}(t,dx)\right\}, 
\end{eqnarray*}
where 
\begin{eqnarray*}
\gamma_{Z}(t)=\gamma\int\mathcal{K}(t-s)\,ds+\int\int x\mathcal{K}(t-s)\left[\mathrm{1}_{\{|x\mathcal{K}(t-s)|\leq1\}}-\mathrm{1}_{\{|x|\leq1\}}\right]\nu(dx)\,ds,
\end{eqnarray*}
\begin{eqnarray*}
\sigma_{Z}^{2}(t)=\sigma^{2}\int\mathcal{K}^{2}(t-s)\,ds
\end{eqnarray*}
and
\begin{eqnarray*}
\nu_{Z}(t,dx)=\int\int\mathrm{1}_{B}\left(x\,\mathcal{K}(t-s)\right)\nu(dx)\,ds,\quad B\in\mathcal{B}(\mathbb{R}).
\end{eqnarray*}
Furthermore,  under our choice of the kernel function $\mathcal{K}$, we can represent the characteristic exponent $\psi$ of the L\'evy-process $(L_{t})$ via the characteristic function $\Phi_\Delta$ of the increments  $Z_{t+\Delta}-Z_t$.  We explicitly derive  for $\Psi_\Delta(u):=\log(\Phi_\Delta(u))$ (see Lemma~\ref{LemmaPhi}),
\begin{equation}
\label{eq:Phi}
\Psi_\Delta(u)=\int_{-\infty}^{\infty}\psi\left(u\left(\mathcal{K_{\alpha}}(x+\Delta)-\mathcal{K_{\alpha}}(x)\right)\right)\,dx
=(\mathcal{L}_\alpha\psi) (\Delta u)+S_\Delta(u),
\end{equation}
where the operator $\mathcal{L}_\alpha$ is defined as
\begin{eqnarray*}
(\mathcal{L}_\alpha f)(x):=\frac{2}{1-\alpha}\,x^{-\frac{\alpha}{1-\alpha}}\int_{0}^{x}f\left(z\right)z^{\frac{2\alpha-1}{1-\alpha}}\,dz
\end{eqnarray*}
for any locally bounded function $f$ and 
\begin{eqnarray}
\psi(u)  =  \mathrm{i}\gamma u-\frac{1}{2}\sigma^{2}u^{2}+\int_{\mathbb{R}}\left(e^{\mathrm{i}ux}-1-\mathrm{i}ux\mathrm{1}_{\{\left|x\right|\leq1\}}\right)\nu(dx).
 \label{LK1}
\end{eqnarray}
Moreover, if 
\begin{eqnarray}
\label{eq:mom-p}
\int\left|x\right|^{p}\nu(x)\,dx<\infty
\end{eqnarray}
for some natural $p>1$ then the function $S_\Delta$ satisfies  (see Lemma~\ref{SDeltap})
\begin{eqnarray}
\label{eq:Scond}
\lim_{\Delta\to0}\Delta^{-(l+\alpha)}S^{(l)}_\Delta(u/\Delta)=0,\quad l=0,\ldots,p,
\end{eqnarray}
and as a result we have convergence
\begin{eqnarray}
\label{eq:lim-psi}
\Psi_\Delta(u/\Delta)\to \Psi(u):=\mathcal{L}_\alpha\psi (u),\quad u\in \mathbb{R}
\end{eqnarray}
for \(\Delta\to 0.\)
Furthermore by inverting the operator $\mathcal{L}_\alpha,$ we get from \eqref{eq:Phi}
\begin{equation*}
\psi''(u)=(1/\Delta)^2(\mathcal{L}^{-1}_\alpha \Psi_\Delta)''(u/\Delta)-(1/\Delta)^2(\mathcal{L}^{-1}_\alpha S_{\Delta})''(u/\Delta)
\end{equation*}
with
\begin{eqnarray}
\label{eq:L-1}
(\mathcal{L}^{-1}_\alpha f)(x):=\frac{\alpha}{2}f(x)+\frac{1-\alpha}{2}x f'(x).
\end{eqnarray}
On the other hand, under the condition \eqref{eq:mom-p} with $p=2$, we obtain from \eqref{LK1},
\begin{equation*}
\psi''\left(u\right)=-\sigma^{2}-\int_{\mathbb{R}}e^{\mathrm{i}u x}\rho\left(x\right)dx,
\end{equation*}
where $\rho(x):=x^2\nu(x).$
Therefore, we can apply the inverse Fourier transform to get 
\begin{equation}
\rho(x)=-\frac{1}{2\pi\Delta^{2}}\int_{\mathbb{R}}e^{-\mathrm{i}u x}[(\mathcal{L}^{-1}_\alpha \Psi_\Delta)''(u/\Delta)+\Delta^2\sigma^2]\, du-R_{\Delta}(u),
\label{eq:rho}
\end{equation}
where 
\[
R_{\Delta}(u):=\int_{\mathbb{R}}e^{-\mathrm{i}u x}r_\Delta(u)\,du, \quad r_\Delta(u):=(1/\Delta)^2(\mathcal{L}^{-1}_\alpha S_{\Delta})''(u/\Delta).
\]
  In view of \eqref{eq:Scond} and \eqref{eq:L-1}, the term $R_{\Delta}$ is of smaller order in $\Delta$ than the first term in \eqref{eq:rho} and we can consider the limiting case \eqref{eq:lim-psi} in \eqref{eq:rho}. 
\par
In this work we assume  that we observe a discretised (high-frequency) trajectory of the limiting L\'evy process  \(X_0,X_{\Delta},\ldots,X_{n\Delta}\) with  characteristic function $\Phi(u):=\mathsf{E}[\exp(\mathrm{i} uX_1)]=\exp(\Psi(u))=\exp(\mathcal{L}_\alpha\psi (u)).$ This assumption is mainly done to simplify analysis and avoid difficulties related to the time dependence structure of the process \((Z_t).\) Still the main features of the underlying inverse problem (e.g. the structure of the inverse operator  
\(\mathcal{L}^{-1}_\alpha\)) remains reflected in our statistical analysis. An extension to the case where one directly observes the process $(Z_t)_{t\geq 0}$  will also be discussed.  
\par
Let us now describe our estimation procedure. Let $W$ be an integrable kernel function such that 
\[
\int_\mathbb{R}  W(x)\,dx=1,\quad \int_\mathbb{R} |x|^{p+1} |W(x)|\,dx<\infty, \quad \int_{\mathbb{R}}  x^{l} W(x)\,dx=0,\quad l=1,\ldots,p, 
\] and suppose that the Fourier transform $\varphi_{W}$ of \(W\) is   supported in $\left[-1,1\right].$ 
Motivated by \eqref{eq:rho}, we propose to estimate  $\rho$    via the estimator: 
\begin{equation}
\begin{split}\widehat{\rho}_{n}(x) &:= -\frac{1}{2\pi}\int_{\mathbb{R}}e^{-\mathrm{i}u x}\left[(\mathcal{L}^{-1}_\alpha \widehat\Psi)''(u)+\widehat{\sigma}_{n}^{2}\right]\varphi_{W}(uh_n)\,du,
\end{split}
\label{nuest}
\end{equation}
 where $\widehat{\Psi}:=\Delta^{-1}\log(\widehat{\Phi}_{\Delta X})$ with 
\begin{eqnarray*}
\widehat{\Phi}_{\Delta X}(u):=\frac{1}{n}\sum_{j=1}^{n} e^{\mathrm{i}u  (\Delta X)_j},\quad u\in\mathbb{R},
\end{eqnarray*}
 \((\Delta X)_j:=X_{\Delta j}-X_{\Delta (j-1)},\)
 $h_{n}$ is a sequence of positive numbers
(bandwidths) such that $h_{n}\to 0$ as $n\to \infty$,
and $\widehat{\sigma}_{n}^{2}$ is an estimator of $\sigma^{2}.$
Our aim  is to construct confidence bands for the transformed L\'evy density $\rho$ on a compact set $I$ in $\mathbb{R}\setminus \{0\}$ and to prove validity of the proposed confidence bands.  
To this end, we shall use  the Gaussian multiplier (or wild) bootstrap.

\section{Main results}
\label{sec:conf}

\subsection{Construction of confidence bands}
Using the equations \eqref{eq:rho} and \eqref{nuest}, the difference  $\widehat{\rho}_{n}(x) -\rho(x)$ can be represented as
\begin{equation}
\label{ZerlDichte}
\widehat{\rho}_{n}(x)-\rho(x)=\underbrace{\bigl(\widehat{\rho}_{n}(x)-\widetilde{\rho}(x)\bigr)}_{R_n(x)}+\underbrace{\bigl(\widetilde{\rho}(x)-\rho(x)\bigr)}_{I_{\sigma^2_n}(x)+I_{\rho_n}(x)},
 \end{equation}
 where
 \begin{equation}
 \label{rhoTilde}
 \begin{split}
&\widetilde{\rho}(x):=-\frac{1}{2\pi\Delta}\int_{\mathbb{R}}e^{-\mathrm{i}ux}\bigl[(\mathcal{L}^{-1}_\alpha \Psi)''(u)+\Delta\sigma^{2}\bigr]\varphi_{W}(uh_n)\,du,\\
&I_{\sigma^2_n}(x):=-\frac{1}{2\pi}\int_{\mathbb{R}}e^{-\mathrm{i}u x}(\widehat{\sigma}^{2}_{n}-\sigma^{2})\varphi_{W}(uh_n)\,du,\\
&I_{\rho_n}(x):=\bigl[\rho\ast(h_n^{-1}W(\cdot/h_n))\bigr](x)-\rho(x).
\end{split}
 \end{equation}
Later we show that under suitable assumptions (Assumption~\ref{assumption1}), the   terms $I_{\rho_n}$ and $I_{\sigma^2_n}$ are asymptotically (as \(n\to \infty\) and \(\Delta\to 0\)) smaller than $R_n$ and hence can be neglected  when constructing the confidence interval for the transformed L\'evy density $\rho.$
Further note that 
\begin{eqnarray*}
(\mathcal{L}^{-1}_\alpha \widehat\Psi)''(u)-(\mathcal{L}^{-1}_\alpha \Psi)''(u)&=& Q_{0}(u)\mathbb{D}(u)+Q_{1}(u)\mathbb{D}'(u)
\\
&&+Q_{2}(u)\mathbb{D}''(u)+Q_{3}(u)\mathbb{D}'''(u),
\end{eqnarray*}
where $\mathbb{D}(u):=\widehat{\Phi}_{\Delta X}(u)-\Phi_{\Delta X}(u),$ \(\Phi_{\Delta X}(u):=\exp(\Delta\mathcal{L}_\alpha\psi (u))\)  and
\begin{equation}
\label{Qm0}
\begin{split}
Q_{0}(u)&=\frac{1}{\Phi_{\Delta X}(u)}\left(\frac{2-\alpha}{2}\left(2\left(\frac{\Phi'_{\Delta X}(u)}{\Phi_{\Delta X}(u)}\right)^{2}-\frac{\Phi''_{\Delta X}(u)}{\Phi_{\Delta X}(u)}\right)\right.\\
&\left.\hspace{1cm}+\frac{1-\alpha}{2} u\left(-\frac{\Phi'''_{\Delta X}(u)}{\Phi_{\Delta X}(u)}+6\frac{\Phi''_{\Delta X}(u)\Phi'_{\Delta X}(u)}{(\Phi_{\Delta X}(u))^2}-4\left(\frac{\Phi'_{\Delta X}(u)}{\Phi_{\Delta X}(u)}\right)^{3}\right)\right)\\
&=\frac{\Delta}{\Phi_{\Delta X}(u)}\left(\frac{2-\alpha}{2}\left(\Delta(\Psi'(u))^{2}-\Psi''(u)\right)\right.\\
&\left.\hspace{2cm}-u\frac{1-\alpha}{2}\left(\Psi'''(u)-3\Delta\Psi''(u)\Psi'(u)+\Delta^2(\Psi'(u))^3\right)\right),
\end{split}
\end{equation}
\begin{equation}
\label{Qm1}
\begin{split}
Q_{1}(u)&=\frac{1}{\Phi_{\Delta X}(u)}\left(3u\left(1-\alpha\right)\left(\frac{\Phi'_{\Delta X}(u)}{\Phi_{\Delta X}(u)}\right)^{2}-3u\frac{\Phi''_{\Delta X}(u)}{\Phi_{\Delta X}(u)}\frac{1-\alpha}{2}-(2-\alpha)\frac{\Phi'_{\Delta X}(u)}{\Phi_{\Delta X}(u)}\right)\\
	&=\frac{\Delta}{\Phi_{\Delta X}(u)}\left(3u\frac{1-\alpha}{2}\left(\Delta(\Psi'(u))^{2}-\Psi''(u)\right)-(2-\alpha)\Psi'(u)\right),
\end{split}
\end{equation}
\begin{equation}
\label{Qm2}
\begin{split}
Q_{2}(u)	&=\frac{1}{\Phi_{\Delta X}(u)}\left(\frac{2-\alpha}{2}-3u\frac{\Phi'_{\Delta X}(u)}{\Phi_{\Delta X}(u)}\frac{1-\alpha}{2}\right)\\
&=\frac{1}{\Phi_{\Delta X}(u)}\left(\frac{2-\alpha}{2}-3u\Delta\Psi'(u)\frac{1-\alpha}{2}\right),
\end{split}
\end{equation}
\begin{equation}
\label{Qm3}
Q_{3}(u)=\frac{1}{\Phi_{\Delta X}(u)}\left(u\frac{1-\alpha}{2}\right).
\end{equation}
With the above notations $R_n(x)$ becomes
\begin{equation}
\label{deltanu}
\begin{split}
R_n(x)&=-\frac{1}{2\pi\Delta}\int_{\mathbb{R}}e^{-\mathrm{i}u x}\bigl[Q_{0}(u)\mathbb{D}(u)+Q_{1}(u)\mathbb{D}'(u)\bigr.\\
&\bigl.\hspace{2cm}+Q_{2}(u)\mathbb{D}''(u)+Q_{3}(u)\mathbb{D}'''(u)\bigr]\varphi_{W}(uh_n)\, du
\end{split}
\end{equation}
or alternatively
\begin{equation}
\label{Rn}
\begin{split}
R_{n}(x)&=\frac{1}{n\Delta}\stackrel[m=0]{3}{\sum}\biggl(\stackrel[j=1]{n}{\sum}\bigl\{ \mathrm{i}^{m}(\Delta X)_{j}^{m}K_{m,n}(x-(\Delta X)_{j})\bigr.\biggr.\\
&\biggl.\bigl.\hspace{3cm}-\mathsf{E}\bigl[\mathrm{i}^{m}(\Delta X)_{1}^{m}K_{m,n}(x-(\Delta X)_{1})\bigr]\bigr\}\biggr),
\end{split}
\end{equation}
where the kernel functions $K_{m,n}(z),\,m=0,1,2,3,$ are defined as
\begin{equation*}
K_{m,n}(z):=-\frac{1}{2\pi}\int_{\mathbb{R}}e^{-\mathrm{i}uz}Q_{m}(u)\varphi_{W}(uh_{n})\,du.
\end{equation*}
The representation \eqref{Rn} is crucial for our analysis. Consider now the process
\begin{equation}
T_{n}(x):= \frac{\sqrt{n}\Delta}{s(x)}R_{n}(x),
\label{Tn}
\end{equation}
where $s^{2}(x)$ is given by
\begin{equation}
\label{VarRn}
s^{2}(x):= \mathsf{Var}\biggl[\stackrel[m=0]{3}{\sum}\mathrm{i}^{m}(\Delta X)_{1}^{m}K_{m,n}(x-(\Delta X)_{1})\biggr]
\end{equation}
 Under some conditions,  we shall show that there exists a tight $\ell^{\infty}(I)$-sequence 
of Gaussian random variables $T_{n}^{G}$  with zero mean  and the same covariance function as one of $\,T_{n}$, and such that the distribution of $\Vert  T_{n}^{G}\Vert_{I}:=\sup_{x\in I}|T_{n}^{G}(x)|$
asymptotically approximates the distribution of $\Vert T_{n}\Vert_{I}$
in the sense that
\begin{equation*}
\underset{z\in\mathbb{R}}{\sup}\bigl|\mathrm{P}\left\{ \bigl\Vert T_{n}\bigr\Vert _{I}\leq z\right\} -\mathrm{P}\left\{ \bigl\Vert T_{n}^{G}\bigr\Vert_{I}\leq z\right\} \bigr|\to 0, \quad n\to \infty.
\end{equation*}
Accordingly, the construction of confidence
bands reduces to  estimating the quantiles of the r.v. $\Vert T_{n}^{G}\Vert_{I}$. To this end we shall use bootstrap. 
Define
\begin{equation*}
\begin{split}c_{n}^{G}(1-\tau)  :=\inf\bigl\{z\in\mathbb{R}:\ \mathrm{P}\left\{\bigl\Vert T_{n}^{G}\bigr\Vert_{I}\leq z\bigr\} \geq1-\tau\right\} 
\end{split}
\end{equation*}
 for $\tau\in(0,1),$ then the \(1-\tau\)-confidence band for \(\rho\) is of the form:
 \begin{equation*}
\widehat{\mathcal{C}}_{1-\tau}(x)=\left[\widehat{\rho}_{n}(x)-\frac{s(x)}{\sqrt{n}\Delta}c_{n}^{G}(1-\tau),\,\widehat{\rho}_{n}(x)+\frac{s(x)}{\sqrt{n}\Delta}c_{n}^{G}(1-\tau)\right],\quad x\in I.
\end{equation*}
Since $\rho(x)\in\widehat{\mathcal{C}}_{1-\tau}(x)$ for all $x\in I$  means that 
\begin{equation*}
\left\Vert \frac{\sqrt{n}\Delta(\widehat{\rho}_{n}(\cdot)-\rho(\cdot))}{s(\cdot)}\right\Vert_{I} \leq c_{n}^{G}(1-\tau),
\end{equation*}
we can show that
\begin{equation*}
\mathrm{P}\left\{\rho(x)\in\widehat{\mathcal{C}}_{1-\tau}(x),\quad \forall \, x\in I\right\} =\mathrm{P}\bigl\{\bigl\Vert T_{n}^{G}\bigr\Vert_{I}\leq c_{n}^{G}(1-\tau)\bigr\}+o(1)
\end{equation*}
as \(n\to \infty.\)
Hence $\widehat{\mathcal{C}}_{1-\tau}(x)$ is a valid confidence band for $\rho$ on $I$ with an approximate level 
$1-\tau$. However, we still need to estimate the quantile $c_{n}^{G}(1-\tau).$
In what follows we consider  the
Gaussian multiplier (or wild) bootstrap to estimate the quantile $c_{n}^{G}(1-\tau)$. \subsection*{\emph{Gaussian multiplier bootstrap.}}
The main idea of the Gaussian multiplier bootstrap consists in reweighting estimated influence functions using
mean zero and unit variance pseudo-random variables, see, e.g.    \cite{chernozhukov2013gaussian} for more details. On the one hand, the advantage of this method  compared to the conventional bootstrap is that we can avoid recomputing the estimator in each bootstrap repetition, and as a result we reduce the calculation time. On the other hand, one of the disadvantages of the Gaussian multiplier bootstrap is that it is necessary to obtain an analytical expression for the corresponding influence  function.       
In our case, this method will be used as follows.
First we simulate $N$ independent centred Gaussian random variables  $\omega_{1},\ldots\omega_{n}\sim N(0,1)$, 
independent of the data $\mathsf{D}_{n}=\bigl\{(\Delta X)_j\bigr\}_{j=0}^{n}$ and  
construct the multiplier process $\widehat{T}_{n}^{MB}(x)$ of the form: 
\begin{equation}
\begin{split}
\widehat{T}_{n}^{MB}(x):=& \frac{1}{\widehat{s}_{n}(x)\sqrt{n}}\biggl(\stackrel[m=0]{3}{\sum}\biggl(\stackrel[j=1]{n}{\sum}\omega_{j}\bigl\{ \mathrm{i}^{m}(\Delta X)_{j}^{m}\widehat{K}_{m,n}(x-(\Delta X)_{j})\bigr.\biggr.\biggr.\\
&\biggl.\biggl.\bigl.\hspace{2cm}-n^{-1}\stackrel[j'=1]{n}{\sum}\mathrm{i}^{m}(\Delta X)_{j'}^{m}\widehat{K}_{m,n}(x-(\Delta X)_{j'})\bigr\} \biggr)\biggr),
\label{TnMB}
\end{split}
\end{equation} 
where
 \begin{equation}
\widehat{s}^{2}_{n}(x):=\mathsf{Var}\biggl[\stackrel[m=0]{3}{\sum}\mathrm{i}^{m}(\Delta X)_{1}^{m}\widehat{K}_{m,n}(x-(\Delta X)_{1})\biggr], \quad x\in I,
\label{sigmaest}
\end{equation}
\begin{equation*}
\widehat{K}_{m,n}(z)=-\frac{1}{2\pi}\int_{\mathbb{R}}e^{-\mathrm{i}uz}\widehat{Q}_{m}(u)\varphi_{W}(uh_n)\,du
\end{equation*}
and  $\widehat{Q}_{m}(u)$ is based on a bootstrapped  version of the empirical characteristic function \(\widehat{\Phi}_{\Delta X}\):
\begin{eqnarray*}
\widehat{\Phi}^{MB}_{\Delta X}(u):=\frac{1}{n}\sum_{j=1}^{n} \omega_j e^{\mathrm{i}u (\Delta X)_j},\quad u\in\mathbb{R}.
\end{eqnarray*}
 Furthermore, we estimate $c_{n}^{G}(1-\tau)$ using quantile 
$\widehat{c}_{n}^{MB}(1-\tau)$ of the distribution of $\Vert \widehat{T}_{n}^{MB}\Vert _{I},$
conditional on the data $\mathsf{D}_{n}.$ The latter quantity
can be computed via simulations. As a result, the confidence band
takes the form
\begin{equation}
\label{KonfIntMB}
\widehat{\mathcal{C}}_{1-\tau}^{MB}(x):=\biggl[\widehat{\rho}_{n}(x)-\frac{\widehat{s}_{n}(x)}{\sqrt{n}\Delta}\widehat{c}_{n}^{MB}(1-\tau),\widehat{\rho}_{n}(x)+\frac{\widehat{s}_{n}(x)}{\sqrt{n}\Delta}\widehat{c}_{n}^{MB}(1-\tau)\biggr],\quad x\in I.
\end{equation}

\subsection{Validity of bootstrap confidence bands.} 

 In this section, we will present the main result, which proves the validity of the confidence band $\widehat{\mathcal{C}}^{MB}_{1-\tau}(x)$ .
\begin{assumption}
\label{assumption1}
We assume that the following conditions are fulfilled.
\begin{enumerate}[label=(\roman*)] 
\item $\int_{\mathbb{R}}|x|^{6+\varsigma}\nu(x)\,dx<\infty$
for some $\varsigma\in[0,1].$
\label{ass1}
\item  Let $r>0$ and let $p$ be an integer such that $p<r\leq p+1$.
The function $\rho$ is p-times differentiable, and $(\rho)^{p}$
is $(r-p)$ -Hölder\footnote{ The function $f\ :\ \mathbb{\mathbb{R}\to R}$ is called $\alpha-$Hölder continuous for $\alpha\in(0,1]$, if 

$\underset{x,y\in\mathbb{R},x\neq y}{\sup}\frac{|f(y)-f(x)|}{|y-x|^{\alpha}}<\infty.$} continuous.
\label{ass3}
\item It holds $h_n^{3}\gtrsim\Delta,\,$  $h_n^{r+1}\Delta^{1/2}n^{1/2}(\log h_n^{-1})^{-1}\to 0$ and $(n\Delta h_n^{3})^{-1/2}(\log n)^{1/2}\to 0$.
\label{ass5}
\item  The estimator $\widehat{\sigma}^{2}$ satisfies
\label{ass6}
\begin{equation*}
\bigl|\sigma^{2}-\widehat{\sigma}^{2}\bigr|\cdot\bigl\Vert  h_n^{-1}W(\cdot/h_n)\bigr\Vert _{I}=o_{\mathrm{P}}\bigl(\Delta^{-1/2} h_n^{-1}n^{-1/2}\log h_n^{-1}\bigr).
\end{equation*}
\end{enumerate}
\end{assumption}
\paragraph{\textbf{Discussion}}
Condition \ref{ass1} is a moment condition and is equivalent to finiteness of (6+)-th moment of the increments process $X_{t+\Delta}-X_t$ (see Lemma~\ref{Moments} for more details). %Condition \ref{ass2} is needed to bound the variance $s^2(x)$ of the empirical process $R_n(x)$ from below. The detailed analysis of this assumption for the increment limiting L\'evy process $(\Delta X)_t$ is carried out in Lemma~\ref{InfPDelta}.
 And finally, Condition \ref{ass6} guarantees that the term $\bigl|\sigma^{2}-\widehat{\sigma}^{2}\bigr|\cdot\bigl\Vert  h_n^{-1}W(\cdot/h_n)\bigr\Vert _{I}$ is of smaller order as compared  to the order  of the leading term in $\widehat{\rho}_n(x)-\rho(x).$
 \par
 Now we
formulate the main theorem of this section, which shows the  convergence of the proposed Gaussian approximation.

\begin{thm} (Gaussian approximation)
\label{thm1}

Under our assumptions, for sufficiently large n, there exists a tight Gaussian random variable $T_{n}^{G}$ in $\ell^{\infty}(I)$ with  zero mean and covariance function of the form \(\frac{W(x,y)}{s(x)s(y)},\) where  
\begin{multline*}
 W(x,y):=\mathsf{I}_\Delta[B_1(x,\cdot)B_1(y,\cdot)]-\mathsf{I}_\Delta[B_1(x,\cdot)]\mathsf{I}_\Delta[B_1(y,\cdot)]
 \\
 +\mathsf{I}_\Delta[B_2(x,\cdot)B_2(y,\cdot)]-\mathsf{I}_\Delta[B_2(x,\cdot)]\mathsf{I}_\Delta[B_2(y,\cdot)]\\
 +\mathrm{i}\bigl\{\mathsf{I}_\Delta[B_1(y,\cdot)B_2(x,\cdot)]-\mathsf{I}_\Delta[B_1(x,\cdot)B_2(y,\cdot)]\bigr\}
 \\
 +\mathrm{i}\bigl\{\mathsf{I}_\Delta[B_1(x,\cdot)]\mathsf{I}_\Delta[B_2(y,\cdot)]-\mathsf{I}_\Delta[B_1(y,\cdot)]\mathsf{I}_\Delta[B_2(x,\cdot)]\bigr\},
\end{multline*}
 the integral operator \(\mathsf{I}_\Delta\) is defined as \(\mathsf{I}_\Delta [f]:=\int f(\upsilon)\,P_{\Delta}(d\upsilon),\)  $s(x)=\sqrt{s^2(x)}$ has the form (\ref{VarRn}) and
  \begin{equation*}
 \begin{split}
 B_{1}(x,\upsilon)&:=K_{0,n}(x-\upsilon)+\upsilon^{2}K_{2,n}(x-\upsilon),\\
 {B}_{2}(x,\upsilon)&:=\upsilon K_{1,n}(x-\upsilon)-\upsilon^{3}K_{3,n}(x-\upsilon).
 \end{split}
 \end{equation*}
Moreover it holds 
  \begin{equation*}
\underset{z\in\mathbb{R}}{\sup}\left|\mathrm{P}\left\{\left \Vert \frac{\sqrt{n}\Delta}{s(\cdot)}(\widehat{\rho}_{n}(\cdot)-\rho(\cdot))\right\Vert_{I}\leq z\right\} -\mathrm{P}\left\{\left \Vert T_{n}^{G}\right\Vert_{I}\leq z\right\} \right|\to 0
\end{equation*}
as \(n\to \infty\) and
\begin{equation}
\label{DifT}
\bigl|\bigl\Vert T_{n}\bigr\Vert _{I}- \bigl\Vert T_{n}^{G}\bigr\Vert_{I}\bigr|=o_{\mathrm{P}} \left(h_n^{1/2}\log h_n^{-1}\right),\quad n\to \infty.
\end{equation}
\end{thm}
Building on Theorem \ref{theorem2}, the following result formally establishes the asymptotic validity of the multiplier bootstrap confidence band $\widehat{\mathcal{C}}^{MB}_{1-\tau}(x)$. 
\begin{thm} (Validity of bootstrap confidence bands).
Under Assumption~\ref{assumption1} we have that 
\begin{equation*}
\mathrm{P}\bigl\{ \rho(x)\in\widehat{\mathcal{C}}^{MB}_{1-\tau}(x),\quad\forall\, x\in I\bigr\} \to1-\tau
\end{equation*}
as $n\to\infty$. Moreover the supremum width of the confidence band of  $\widehat{\mathcal{C}}^{MB}_{1-\tau}(x)$ is  of order $\mathcal{O}_{p}\bigl((n\Delta h_n^3)^{-1/2}\sqrt{\log n}\bigr)$.
\label{theorem2}
\end{thm}
\subsection*{Discussion on choosing for $\Delta$, $n$ and $h_n$}

From the lemma \ref{lemvar} applies $\inf_{x\in I}s^{2}(x)\gtrsim\Delta h_{n}^{-3}$, which leads to the first assumption \ref{assumption1} \ref{ass5}, namely $h_{n}^{3}\gtrsim\Delta$. According to the representation \ref{ZerlDichte} applies
\begin{equation*}
\widehat{\rho}_{n}(x)-\rho(x)=\underbrace{\bigl(\widehat{\rho}_{n}(x)-\widetilde{\rho}(x)\bigr)}_{R_n(x)}+\underbrace{\bigl(\widetilde{\rho}(x)-\rho(x)\bigr)}_{I_{\sigma^2_n}(x)+I_{\rho_n}(x)}.
\end{equation*} 
Under the condition of the dominance of the convergence rate of the first term $\Vert R_{n}(x)\Vert_{\mathbb{R}}=\mathcal{O}\left(\Delta^{-1/2}h_{n}^{-1}n^{-1/2}\log h_{n}^{-1}\right)$ follows assumption \ref{assumption1} \ref{ass6}, namely  $\bigl|\sigma^{2}-\widehat{\sigma}^{2}\bigr|\cdot\bigl\Vert h_{n}^{-1}W(\cdot/h)\bigr\Vert_{I}=o_{\mathrm{P}}\bigl(\Delta^{-1/2}h_{n}^{-1}n^{-1/2}\log h_{n}^{-1}\bigr)$ and assumption \ref{assumption1} \ref{ass5}, namely $h_{n}^{r+1}\Delta^{1/2}n^{1/2}(\log h_{n}^{-1})^{-1}\to0$. If the two terms are supposed to be significant, the last condition is represented in the form $h_{n}^{r+1}\Delta^{1/2}n^{1/2}(\log h_{n}^{-1})^{-1}\gtrsim 1$, which leads to a relationship between $n$ and $h_n$. Let $\log h_{n}^{-1}\lesssim\log n\lesssim n^{\varepsilon}$, where $\varepsilon\to 0$ then applies
\begin{equation*}
\begin{split} 
& h_{n}^{r}\lesssim\Delta^{-1/2}h_{n}^{-1}n^{-1/2}\log h_{n}^{-1}\\
 & h_{n}^{r}\lesssim h_{n}^{-5/2}n^{-1/2}\log h_{n}^{-1}\\
 & h_{n}^{r+5/2}\lesssim n^{-1/2+\varepsilon}\\
 &h_{n}\lesssim n^{-\frac{1-2\varepsilon}{2r+5}}\\
 & n\lesssim h_{n}^{-\frac{2r+5}{1-2\varepsilon}}.
\end{split}
\end{equation*} 
From the proof of the theorem \ref{thm1} it follows that
\begin{equation*}
\frac{\sqrt{n}\Delta(\widehat{\rho}_{n}(x)-\rho(x))}{s(x)}=T_{n}(x)+o_{\mathrm{P}}\left(h_{n}^{1/2}\log h_{n}^{-1}\right)
\end{equation*}
and 
\begin{equation*}
\bigl|\bigl\Vert\mathsf{G}_{n}\bigr\Vert_{\mathcal{F}_{n}}-V_{n}\bigr|=\mathcal{O}_{P}\biggl\{\frac{(\log n)^{1+1/q}}{n^{1/2-1/q}\sqrt{\Delta h_{n}^{-1}}}+\frac{\log n}{(n\Delta h_{n}^{-1})^{1/6}}\biggr\}=\mathcal{O}_{P}\biggl\{\frac{\log n}{(n\Delta h_{n}^{-1})^{1/6}}\biggr\}
\end{equation*}
further follows
\begin{equation*}
\begin{split} & h_{n}^{1/2}\log h_{n}^{-1}\gg\frac{\log n}{(n\Delta h_{n}^{-1})^{1/6}}\\
 & h_{n}^{1/2}\log h_{n}^{-1}(n\Delta h_{n}^{-1})^{1/6}\left(\log n\right)^{-1}\to\infty.
\end{split}
\end{equation*}
We also find the relationship between $n$ and $h_{n}$ so that the error of the Gaussian approximation $\bigl|\bigl\Vert\mathsf{G}_{n}\bigr\Vert_{\mathcal{F}_{n}}-V_{n}\bigr|$ is comparable to the approximation error $\Vert T_{n}(x)\Vert_{\mathbb{R}}$. Let $\log h_{n}^{-1}\lesssim\log n\lesssim n^{\varepsilon}$, where $\varepsilon\to 0$, then applies
\begin{equation*}
\begin{split} 
& h_{n}^{1/2}\log h_{n}^{-1}\gtrsim\frac{\log n}{(n\Delta h_{n}^{-1})^{1/6}}\\
 & h_{n}^{1/2}\log h_{n}^{-1}\gtrsim\frac{\log n}{(nh_{n}^{2})^{1/6}}\\
 & h_{n}^{5/6}\log h_{n}^{-1}\gtrsim n^{-1/6}\log n\\
 &h_{n}\gtrsim n^{-1/5}.
\end{split}
\end{equation*}
Furthermore, it should be noted that the bootstrap approximation $\Vert T_{n}^{MB}(x)\Vert_{\mathbb{R}}$ of a Gaussian process $\Vert T_{n}^{G}(x)\Vert_{\mathbb{R}}$ according to the theorem \ref{theorem2} has the order
\begin{equation*}
\bigl|\bigl\Vert\mathsf{G}_{n}^{\xi}\bigr\Vert_{\mathcal{F}_{n}}-V_{n}^{\xi}\bigr|=\mathcal{O}_{P}\biggl\{\frac{(\log n)^{2+1/q}}{n^{1/2-1/q}\sqrt{\Delta h_{n}^{-1}}}+\frac{(\log n)^{7/4+1/q}}{(n\Delta h_{n}^{-1})^{1/4}}\biggr\}=\mathcal{O}_{P}\biggl\{\frac{(\log n)^{7/4+1/q}}{(n\Delta h_{n}^{-1})^{1/4}}\biggr\}.
\end{equation*} 
Therefore the rate of convergence of this approximation is faster than the one mentioned above in the theorem \ref{thm1} order, namely applies
\begin{equation*}
 h_{n}^{1/2}\log h_{n}^{-1}\gg\frac{\log n}{(n\Delta h_{n}^{-1})^{1/6}}\gg\frac{(\log n)^{7/4+1/q}}{(n\Delta h_{n}^{-1})^{1/4}}.
\end{equation*}
It is also important to note that according to the theorem \ref{theorem2}, the the supremum width of the confidence band should also converge to 0:
\begin{equation*}
(n\Delta h_{n}^{3})^{-1/2}\sqrt{\log n}\longrightarrow0.
\end{equation*}
Since $h_{n}^{-3/2}\left(\log n\right)^{3}\left(\log h_{n}^{-1}\right)^{-3}\ll h^{-2}\sqrt{\log n}$
then applies
\begin{equation*}
\begin{split} 
&  h_{n}^{-1/2}(n\Delta h_{n}^{-1})^{-1/6}\log n\left(\log h_{n}^{-1}\right)^{-1}\to0,\\
 & h_{n}^{1/2}\log h_{n}^{-1}(n\Delta h_{n}^{-1})^{1/6}\left(\log n\right)^{-1}\to\infty
\end{split}
\end{equation*} 
if assumption $(n\Delta h_{n}^{3})^{-1/2}\sqrt{\log n}\to0$ satisfies. Since the expression $(n\Delta h_{n}^{3})^{-1/2}\sqrt{\log n}\to0$  has a slower order of convergence than the expression $h_{n}^{-1/2}(n\Delta h_{n}^{-1})^{-1/6}\log n\left(\log h_{n}^{-1}\right)^{-1}\to0$, then the expression  $h_{n}\gtrsim n^{-1/5}$ should be specified:
\begin{equation*}
 \begin{split} & (n\Delta h_{n}^{3})^{-1/2}\sqrt{\log n}\to0\\
 & (nh_{n}^{6})^{-1/2}\sqrt{\log n}\to 0\\
 & h_{n}^{-3}\lesssim n^{1/2\left(1-\varepsilon\right)}\\
 & h_{n}\gtrsim n^{-1/6\left(1-\varepsilon\right)}\\
 & n\gtrsim h_{n}^{-6/\left(1-\varepsilon\right)},
\end{split}
 \end{equation*}
 where $\log n\ll n^{\varepsilon}, \varepsilon\to 0.$ 
 The  supremum width of the confidence band is minimal if $n\approx h_{n}^{-\frac{2r+5}{1-2\varepsilon}}$, da $(n\Delta h_{n}^{3})^{-1/2}\sqrt{\log n}\to\min\Longleftrightarrow h_{n}^{-6}\log n/n\to\min$ applies.
 Then the following applies to the relationship between $n$ and $h_{n}$:
 \begin{equation*}
\begin{split}
&h_{n}^{-6/\left(1-\varepsilon\right)}\leq n\leq h_{n}^{-\frac{2r+5}{1-2\varepsilon}}\\
&n^{-1/6\left(1-\varepsilon\right)}\leq h_{n}\leq n^{-\frac{1-2\varepsilon}{2r+5}}.
\end{split}
\end{equation*} 
\section{Numerical results}
\label{sec:num}
Consider the integral (\ref{Zt}) with the kernel $\mathcal{K}_\alpha$ from the class (\ref{Kalpha}) for some \(\alpha\in (0,1)\), and the L\'evy process \((L_t)\) defined by 
\begin{eqnarray}
\begin{split} & L_{t}=\gamma t+\sigma W_{t}+CPP_{t}^{(1)}\cdot\mathbb{I}\left\{ t\geq0\right\} +CPP_{t}^{(2)}\cdot\mathbb{I}\left\{ t<0\right\} ,\\
 & CPP_{t}^{\left(k\right)}:=\sum_{j=1}^{N_{t}^{\left(k\right)}}Y_{j}^{\left(k\right)},\qquad k=1,2,
\end{split}
\end{eqnarray}
where $\gamma\in\mathbb{R}$ is a drift, $\sigma\geq0$, $W_{t}$ is a Brownian motion, $N_{t}^{(1)}$, $N_{t}^{(2)}$, are two Poisson processes with intensity $\lambda$, $Y_{1}^{(1)}$, $Y_{2}^{(1)}$,... and $Y_{1}^{(2)}$, $Y_{2}^{(2)}$,... are i.i.d. r.v's with an absolutely continuous distribution, and all $Y'$s, $N_{t}^{(1)}$, $N_{t}^{(2)}$, $W_{t}$ are jointly independent.
For simulation study, we take $\gamma=5, \lambda=1$ and $\sigma=0,$ and aim to estimate the corresponding L\'evy density of \((L_t)\) under different choices of the parameter $\alpha,$ namely $\alpha=0.5,$ 0.8 and 0.9.

\textbf{Simulation.} 
Recall that the L\'evy-driven moving average process  $Z_{t}$ satisfying \ref{Zt} is observed at $n$ discrete instants $t_j = j\Delta,\, j = 1,...,n,$ with regular sampling interval and our estimation procedure is based on the random variables $\,(\Delta Z)_j:=Z_{j\Delta}-Z_{(j-1)\Delta},\, j = 1,\ldots,n,$  which are independent, identically distributed, with common characteristic function $\Phi$. We assume that, as $n$ tends to infinity, $\Delta = \Delta_n$ tends to 0 and $n\Delta$ tends to infinity.

For $k=1,2,$ denote the jump times of $L_{t}^{\left(k\right)}$ by $s_{1}^{\left(k\right)}$, $s_{2}^{\left(k\right)}$,...., corresponding to the jump sizes $Y_{1}^{\left(k\right)}$, $Y_{2}^{\left(k\right)}$,.... $Y_{1}^{\left(k\right)}$ and $Y_{2}^{\left(k\right)}$ are independent r.v's with standard exponential distribution with parameter $\lambda$. Note that
\begin{eqnarray}
 \quad\quad Z_{t}=
\begin{cases}
\frac{2\gamma}{1+\alpha}+\underset{j\in J^{\left(1\right)}}{\sum}\left(1-\alpha|t-s_{1}^{\left(j\right)}|\right)^{1/\alpha}Y_{1}^{\left(j\right)}, & \text{if}\;\;t\geq\frac{1}{\alpha}\\
\frac{2\gamma}{1+\alpha}+\underset{j\in J^{\left(2\right)}}{\sum}\left(1-\alpha|t-s_{1}^{\left(j\right)}|\right)^{1/\alpha}Y_{1}^{\left(j\right)}\\
\hspace{2.5cm}+\underset{j\in J^{(3)}}{\sum}\left(1-\alpha|t+s_{2}^{\left(j\right)}|\right)^{1/\alpha}Y_{2}^{\left(j\right)}, & \text{if}\;\;t<\frac{1}{\alpha},
\end{cases}
\end{eqnarray}
where 
\begin{eqnarray*}
\begin{split}
J^{\left(1\right)} & :=\left\{ j:t-\frac{1}{\alpha}\leq s_{1}^{\left(j\right)}\leq t+\frac{1}{\alpha}\right\} ,\\
J^{\left(2\right)} & :=\left\{ j:0\leq s_{1}^{\left(j\right)}\leq t+\frac{1}{\alpha}\right\} ,\\
J^{\left(3\right)} & :=\left\{ j:0\leq s_{2}^{\left(j\right)}\leq\frac{1}{\alpha}-t\right\}.
\end{split}
\end{eqnarray*}
 Finally, the limiting L\'evy process  is defined by $X_j:=(Z_{j\Delta}-Z_{(j-1)\Delta})/\Delta,\, j = 1,...,n.$

Typical trajectory of the of the limiting L\'evy process $X_t:=(\Delta Z)_{t}/\Delta$ is presented in Figure \ref{fig:Zd1}.
\begin{figure}[h]
\center{\includegraphics[scale=0.35]{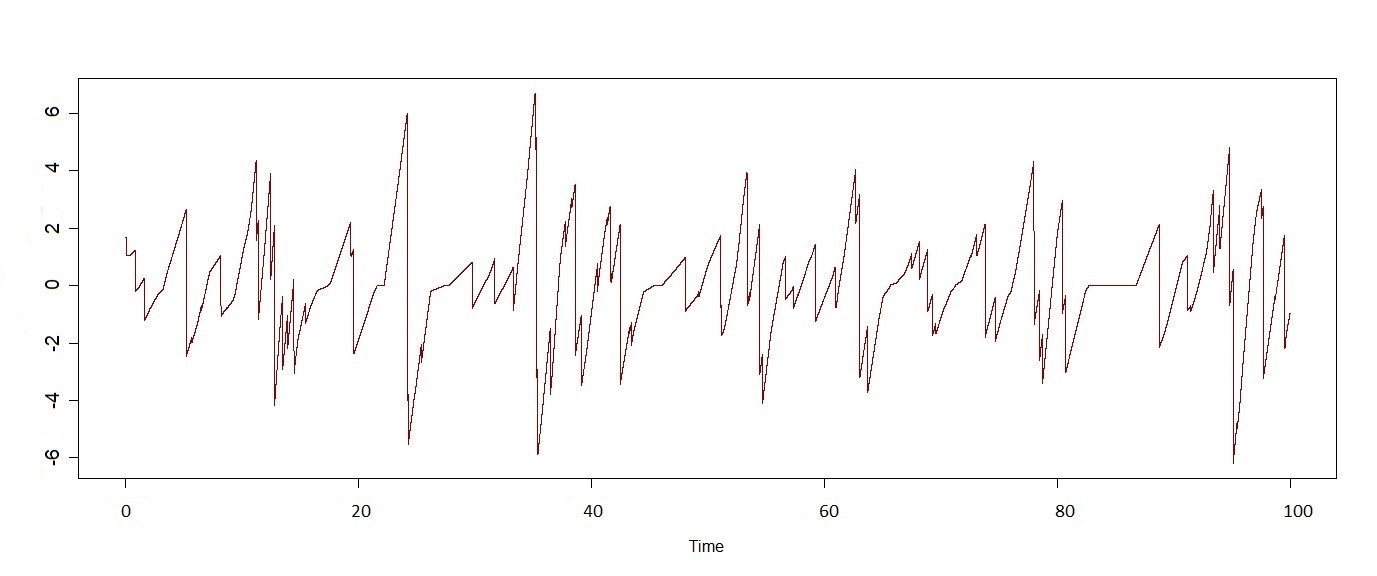}}
\caption{ Typical trajectory of the limiting L\'evy process  $X_t:=(\Delta Z)_{t}/\Delta$ with the value of the parameter $\alpha = 0.5$}
\label{fig:Zd1}
\end{figure}

\textbf{Estimation.}
 Following the ideas from Section 2, we estimate the transformed L\'evy measure by Equation (\ref{nuest}) under different choices of $\alpha$. To show  the convergence properties of the considered estimates, we provide simulations with different values of n. Figure~\ref{fig:CharExpX} shows an estimate of the real part of the characteristic exponent $\psi$ of the L\'evy process $ (L_t) $ through discrete observations of the limit L\'evy process $(X_t).$
 \begin{figure}[h]
\center{\includegraphics[scale=0.4]{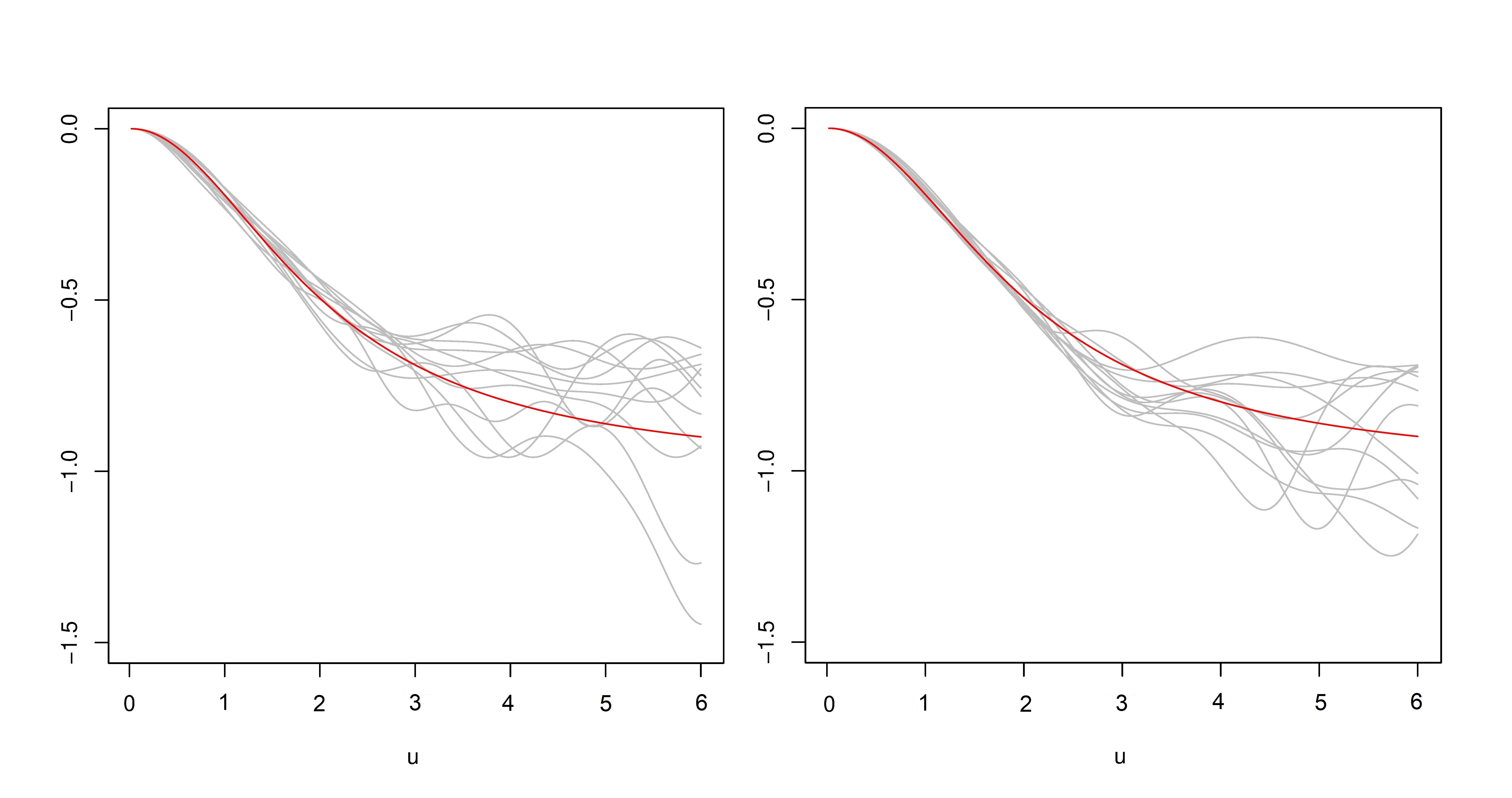}}
\caption{Real part of the characteristic exponent $\psi(u)$ (red) and the real part of 10 realizations of the empirical characteristic exponent $\psi_n(u)$ (gray) with the value of the parameter $\alpha = 0.8$ (left) and $\alpha = 0.9$ (right)}
\label{fig:CharExpX}
\end{figure}
It is important to note that a good estimate of the characteristic exponent $\psi(u)$ is obtained when $ u\in (0,2)$.
Figure \ref{fig:RhonX} shows the estimator of the transformed L\'evy density $\rho$ through discrete observations of the Limit-Lévy process $(X_t)$.
\begin{figure}[h]
\center{\includegraphics[scale=0.3]{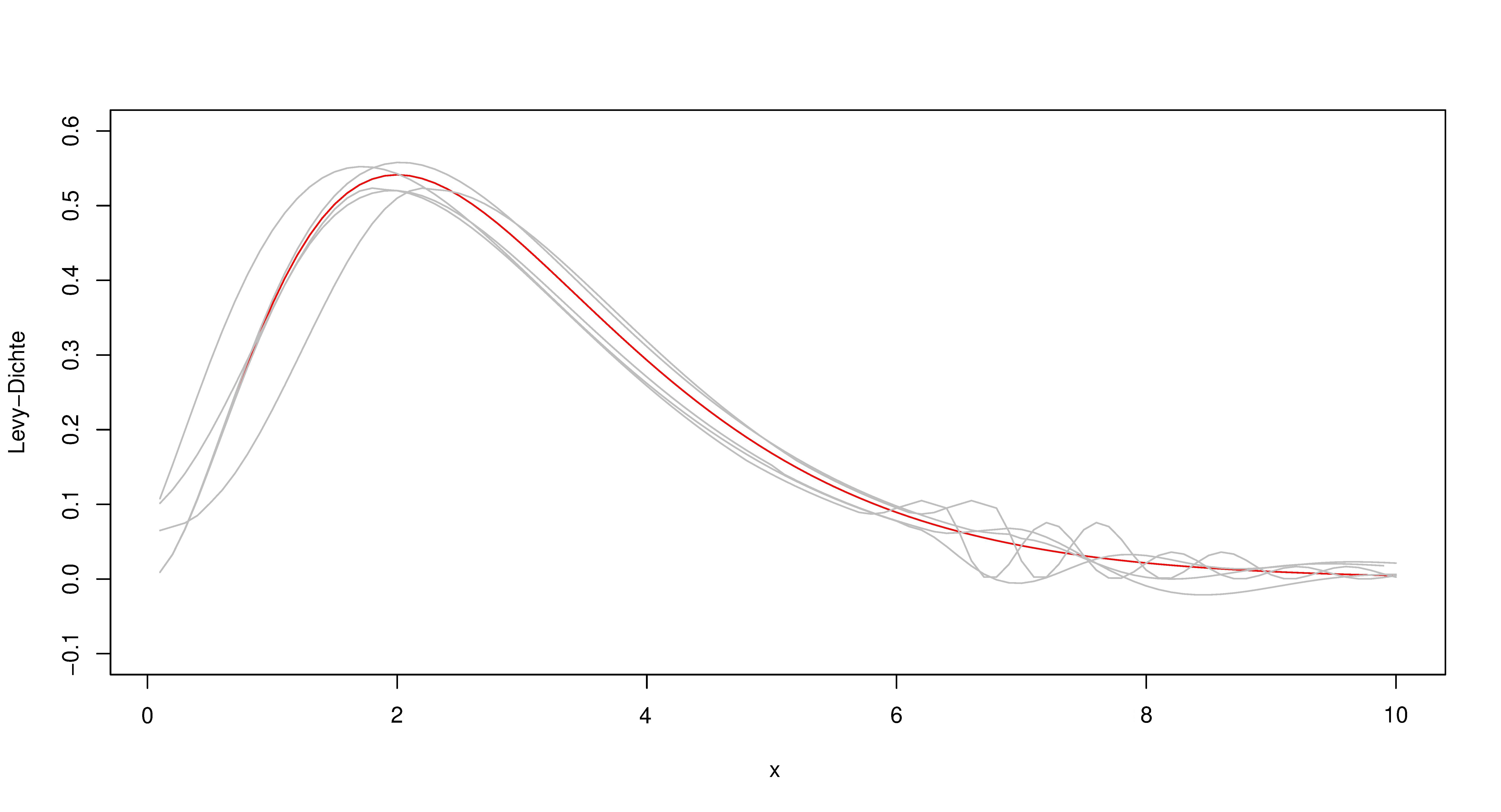}}
\caption{Transformed L\'evy-density $\rho$ (red) and 5 realizations of the estimator of the transformed L\'evy-density $\rho_n$ (gray) with the value of the parameter $\alpha=0.5$}
\label{fig:RhonX}
\end{figure}
  The estimation of the L\'evy densities based on 25 simulation runs are presented in Figures \ref{fig:Nuest}. 

\begin{figure}[h]
\center{\includegraphics[scale=0.25]{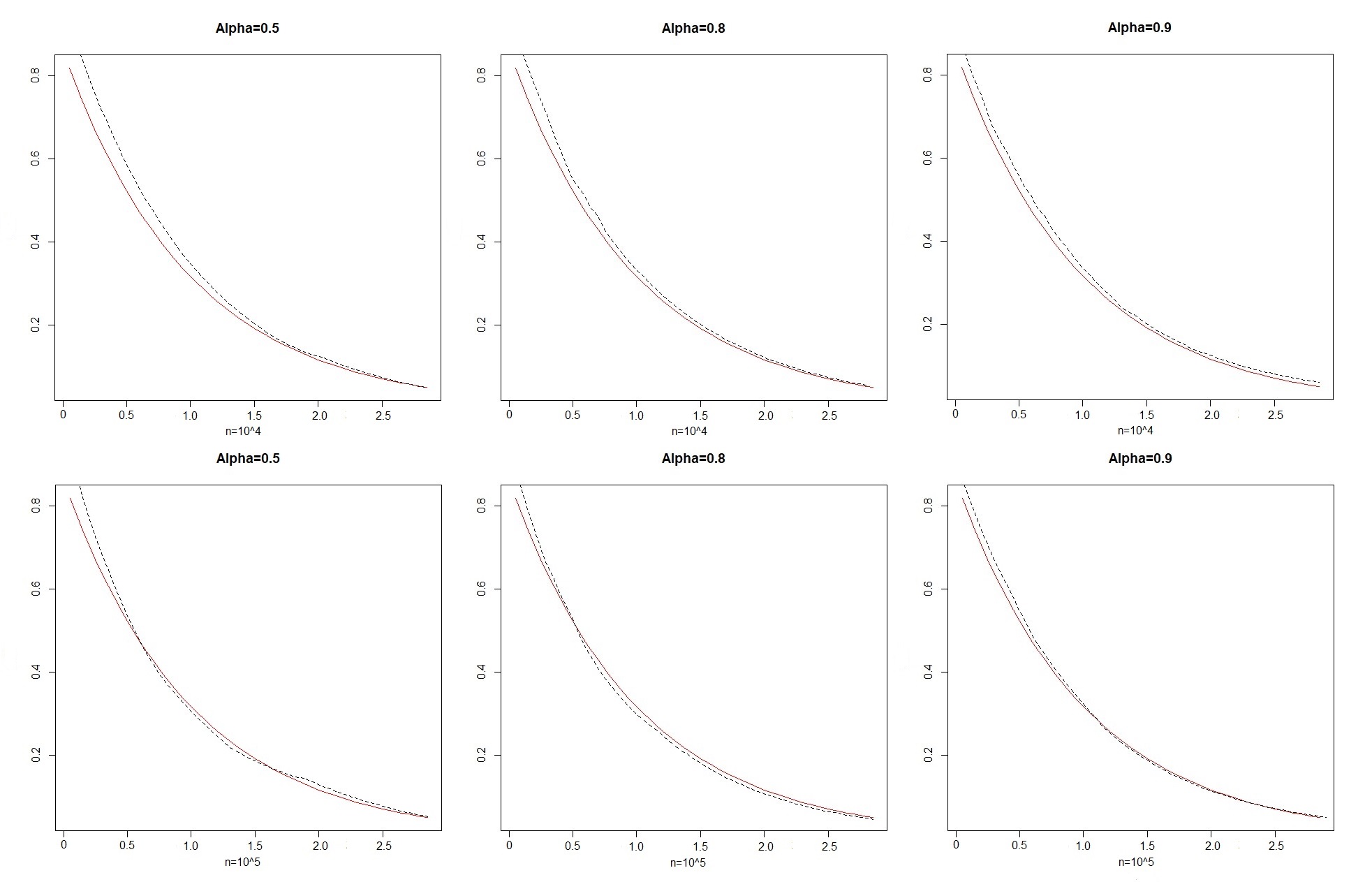}}
\caption{Estimates of the L\'evy densities (dashed lines) for  different values of $n$ and $\alpha$}
\label{fig:Nuest}
\end{figure}
On the one hand, a priori choice for the parameter $h_n$ can be found using the interval for $u$ where the characteristic function of the process $\Delta X$ can be approximated by empirical characteristic function.  On the other hand, a priori choice of the parameter $h_n$ has to consider the assumption \ref{assumption1} \ref{ass5}. Note that the parameter $h_n$ is chosen by numerical optimization. Namely, for each choice of $\alpha$, we first estimate the L\'evy densities for each $h_n$ from an equidistant grid (from \(0.05\) to \(0.5\) with step \(0.05\)), and then analyze the quality of estimation in terms of the minimal mean square error.  Because the best results are obtained for $h_n$ from 0.1 to 0.2, we reproduce the estimation procedure for $h_n$ from another grid (from \(0.08\) to \(0.25\) with step \(0.01\)). After several iterations, we stop the procedure.
It is important to note that in the real-life examples, the aforementioned strategy for choosing $h_n$ should be changed, because the comparison with respect to the mean square error is not possible. One should rather use adaptive methods.
The simulation results illustrated in the figure \ref{fig:Nuest} show that the convergence rates significantly depend on the parameter $\alpha$. More precisely, it turns out that the quality of estimation increases with growing $\alpha$, and the best rates correspond to the case when $\alpha$ is close to 1. This can be explained by the fact that observations become less dependent as $\alpha$ increases.  Let us remark that in  Figure~\ref{fig:Nuest}  we show  the real parts of the estimate $\widehat{\nu}_{n}(x)$. The imaginary part of the considered estimate is quite small (of order $10^{-8}$) and is shown in the Figure~\ref{fig:ImNuest}.
\begin{figure}[h]
\center{\includegraphics[scale=0.3]{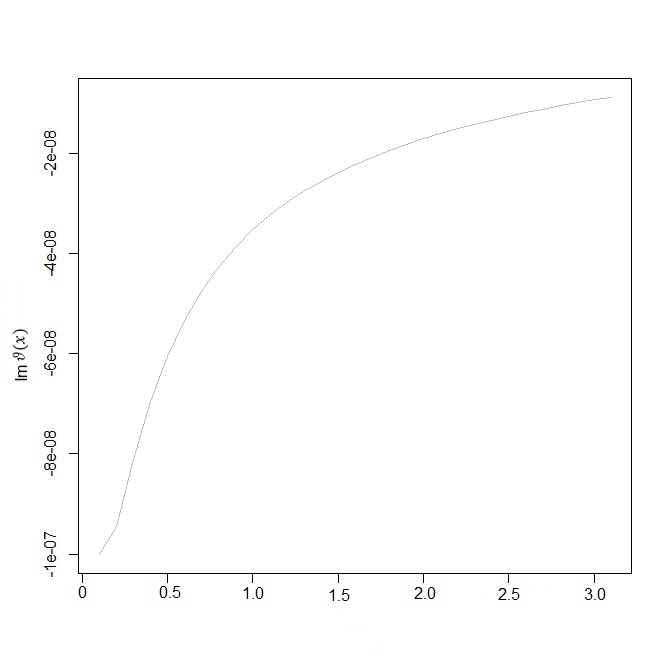}}
\caption{Imaginary part of the estimate of the L\'evy densities  for   $\alpha=0.5$}
\label{fig:ImNuest}
\end{figure}
Finally, following the ideas from Section 2, we construct the confidence interval for the transformed  L\'evy density $\rho$ via the Gaussian multiplier bootstrap method with parameters $\alpha=0.8$, $ n = 10^{5}$ and the confidence level $ 0.9 $. The dashed line in Figure~\ref{fig:confB} represents the  estimator $\widehat{\rho}_n $ of the transformed L\'evy density \(\rho\) (red line).
\begin{figure}[h]
\center{\includegraphics[scale=0.36]{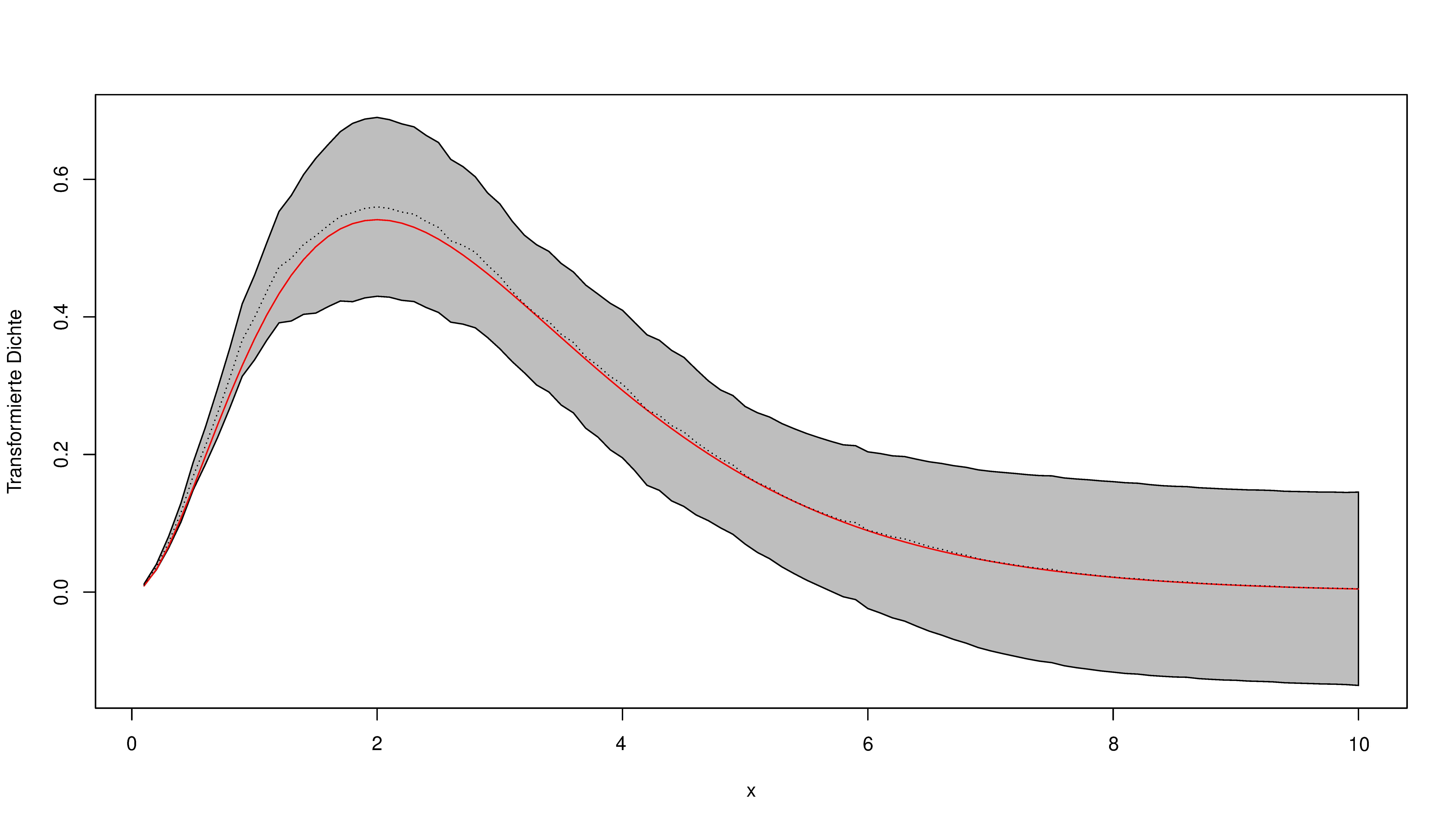}}
\caption{Confidence interval for the transformed L\'evy density $\rho$ via Gaussian multiplier bootstrap method with the parameter $\alpha = 0.8 $, $n=10^{5}$ and the confidence level $0.9$. (The dashed line is the estimator $\widehat{\rho}_n $ of the L\'evy density, the red line is the transformed L\'evy density $\rho$)}
\label{fig:confB}
\end{figure}

\section{Proofs} 
\label{sec:proofs}
For a symmetric kernel $\mathcal{K}_{\alpha}$ of the form (\ref{Kalpha}) we first show  \ref{eq:Phi}.
\begin{lem}
\label{LemmaPhi}
We have 
\begin{equation*}
\Psi_\Delta(u)=\int_{-\infty}^{\infty}\psi\left(u\left(\mathcal{K_{\alpha}}(x+\Delta)-\mathcal{K_{\alpha}}(x)\right)\right)\,dx
=\mathcal{L}_\alpha\psi (\Delta u)+S_\Delta(u),
\end{equation*}
where the operator $\mathcal{L}_\alpha$ is defined as
\begin{eqnarray*}
\mathcal{L}_\alpha f(x):=\frac{2}{1-\alpha}\,x^{-\frac{\alpha}{1-\alpha}}\int_{0}^{x}f\left(z\right)z^{\frac{2\alpha-1}{1-\alpha}}\,dz
\end{eqnarray*}
for any locally bounded function $f$ and 
\begin{eqnarray*}
\psi(u)=\mathrm{i}\gamma u-\frac{1}{2}\sigma^{2}u^{2}+\int_{\mathbb{R}}\left(e^{\mathrm{i}ux}-1-\mathrm{i}ux\mathrm{1}_{\{\left|x\right|\leq1\}}\right)\nu(dx).
\end{eqnarray*}
Furthermore, $S_{\Delta}(u)$ has the form
\begin{equation}
\begin{split}
S_{\Delta}(u)&=S_{1}(u)+S_{2}(u)\\
&=\int_{-1/\alpha}^{1/\alpha}\left[\psi\bigl(u(\mathcal{K_{\alpha}}(x+\Delta)-\mathcal{K_{\alpha}}(x))\bigr)-\psi(u\Delta\mathcal{K_{\alpha}}'(x))\right]\,dx\\
&+\int_{-1/\alpha-\Delta}^{-1/\alpha}\psi(u\mathcal{K_{\alpha}}(x+\Delta))\,dx.
\label{SDelta}
\end{split}
\end{equation}
\end{lem}
\begin{proof}
In the previously described scenario, the characteristic function $\Phi_{\Delta}$ of the increment process $Z_{t+\Delta}-Z_{t}$ has the form
\begin{equation*}
\Phi_{\Delta}(u):=\mathsf{E}\bigl[\exp\bigl(\mathrm{i}u(Z_{t+\Delta}-Z_{t})\bigr)\bigr]=\exp\bigl(\Psi_{\Delta}(u)\bigr),
\end{equation*}
where 
\begin{equation*}
\Psi_{\Delta}(u):=\int_{-\infty}^{\infty}\psi\bigl[u\bigl(\mathcal{K}_{\alpha}(x+\Delta)-\mathcal{K}_{\alpha}(x)\bigr)\bigr]\,dx
\end{equation*}
The last expression can be obtained using Lemma 5.5 in Sato \cite{sato1999levy} and taking into account the fact that
\begin{equation*}
u(Z_{t+\Delta}-Z_{t})=\int_{\mathbb{R}}u\bigl(\mathcal{K}_{\alpha}(x+\Delta)-\mathcal{K}_{\alpha}(x)\bigr)dL_s.
\end{equation*}
We also calculate the first two derivatives of $\mathcal{K}_{\alpha}$, 
\begin{equation*}
\begin{split}
\mathcal{K}'_{\alpha}(x)= & -\left(1-\alpha x\right)^{\frac{1-\alpha}{\alpha}}=-\mathcal{K}_{\alpha}^{1-\alpha}(x),\quad \forall\, x>0\\
\mathcal{K}_{\alpha}''(x)= & (1-\alpha)(1-\alpha x)^{\frac{1-2\alpha}{\alpha}}=(1-\alpha)\mathcal{K}_{\alpha}^{1-2\alpha}(x),\quad \forall\, x>0.
\end{split}
\end{equation*}
Then the characteristic function $\Phi_{\Delta}$ of the increment process $Z_{t+\Delta}-Z_{t}$ has the form:
\begin{equation}
\begin{split}
\Phi_{\Delta}(u) &=\exp\biggl[\int_{-\infty}^{\infty}\psi\bigl(u\bigl(\mathcal{K_{\alpha}}(x+\Delta)-\mathcal{K_{\alpha}}(x)\bigr)\bigr)\,dx\biggr]\\
& =\exp\biggl[2\int_{0}^{1/\alpha}\psi\bigl(u\Delta\mathcal{K}'_{\alpha}(x)\bigr)\,dx+S_{\Delta}(u)\biggr]\\
 & =\exp\biggl[2\int_{0}^{1/\alpha}\frac{\psi\bigl(u\Delta\mathcal{K}'_{\alpha}(x)\bigr)}{\mathcal{K}_{\alpha}''(x)}\,d\mathcal{K}_{\alpha}'(x)+S_{\Delta}(u)\biggr]\\
 & =\exp\biggl[\frac{2}{1-\alpha}\int_{0}^{1}\psi(u\Delta y)y^{\frac{2\alpha-1}{1-\alpha}}\,dy+S_{\Delta}(u)\biggr]\\
 & =\exp\biggr[\frac{2}{1-\alpha}\,(u\Delta)^{-\frac{\alpha}{1-\alpha}}\int_{0}^{u\Delta}\psi(z)z^{\frac{2\alpha-1}{1-\alpha}}\,dz+S_{\Delta}(u)\biggr]\\
 &=\exp\bigl[\mathcal{L}_{\alpha}\psi(u\Delta)+S_{\Delta}(u)\bigr],
\end{split}
\label{Phi}
\end{equation}
where 
\begin{equation}
\begin{split}
S_{\Delta}(u)&=S_{1}(u)+S_{2}(u)\\
&=\int_{-1/\alpha}^{1/\alpha}\left[\psi\bigl(u(\mathcal{K_{\alpha}}(x+\Delta)-\mathcal{K_{\alpha}}(x))\bigr)-\psi(u\Delta\mathcal{K_{\alpha}}'(x))\right]\,dx\\
&+\int_{-1/\alpha-\Delta}^{-1/\alpha}\psi(u\mathcal{K_{\alpha}}(x+\Delta))\,dx.
\label{SDelta}
\end{split}
\end{equation}
\end{proof}
Note that the distribution of $Z_{t+\Delta}-Z_{t}$ is infinitely divisible.
Next, we prove (\ref{eq:Scond}).
\begin{lem}
\label{SDeltap}
 Let
 \begin{equation*}
 \int|x|^{p}\nu(x)\,dx<\infty
 \end{equation*}
 for some $p\in\mathbb{N}.$
 Then applies 
\begin{equation}
\lim_{\Delta\to0}\Delta^{-(p+\alpha)}S_{\Delta}^{(p)}(u/\Delta)=0,
\end{equation} 
where $S_{\Delta}\in C^{p}(\mathbb{R})$ is defined by \ref{SDelta}.
\end{lem}
\begin{proof}
We have
\begin{equation*}
\begin{split}
S_{1}^{(p)}(u)&=\int_{-1/\alpha}^{1/\alpha}\bigl(\mathcal{K_{\alpha}}(x+\Delta)-\mathcal{K_{\alpha}}(x)\bigr)^{p}\psi^{(p)}\bigl(u(\mathcal{K_{\alpha}}(x+\Delta)-\mathcal{K_{\alpha}}(x))\bigr)dx\\
&\hspace{2.5cm}-\int_{-1/\alpha}^{1/\alpha}\bigl(\Delta\mathcal{K_{\alpha}}'(x)\bigr)^{p}\psi^{(p)}\bigl(u\Delta\mathcal{K_{\alpha}}'(x)\bigr)dx\\
&=\int_{-1/\alpha}^{1/\alpha}\bigl[\bigl(\mathcal{K_{\alpha}}(x+\Delta)-\mathcal{K_{\alpha}}(x)\bigr)^{p}-\bigl(\Delta\mathcal{K_{\alpha}}'(x)\bigr)^{p}\bigr]\\
& \hspace{2.5cm}\times\psi^{(p)}\bigl(u\bigl(\mathcal{K_{\alpha}}(x+\Delta)-\mathcal{K_{\alpha}}(x)\bigr)\bigr)dx\\
&+\int_{-1/\alpha}^{1/\alpha}\bigl(\Delta\mathcal{K}_{\alpha}'(x)\bigr)^{p}\bigl[\psi^{(p)}(u(\mathcal{K}_{\alpha}(x+\Delta)-\mathcal{K}_{\alpha}(x)))-\psi^{(p)}(u\Delta\mathcal{K}_{\alpha}'(x))\bigr]dx\\
&:= I_1+I_2.
\end{split}
\end{equation*}
Using the fact that $\bigl\Vert \mathcal{K_{\alpha}}'\bigr\Vert _{\infty}<\infty,$  we derive the integral form of the remainder term in the Taylor's formula: 
\begin{equation}
\label{KDiff}
\begin{split}
\bigl|(\mathcal{K_{\alpha}}(x+\Delta)-\mathcal{K_{\alpha}}(x))^{p}-(\Delta\mathcal{K_{\alpha}}'(x))^{p}\bigr|	&\leq B_{\alpha}\Delta^{p}\int_{x}^{x+\Delta}(x+\Delta-t)\bigl|\mathcal{K_{\alpha}}''(t)\bigr|dt\\
&=B_{\alpha}\Delta^{p}\int_{0}^{\Delta}(\Delta-t)\bigl|\mathcal{K_{\alpha}}''(t+x)\bigr|dt.
\end{split}
\end{equation}
If $\int_{\mathbb{R}}\bigl|\mathcal{K_{\alpha}}''(t)\bigr|\,dt<\infty,$ then we have 
\begin{equation}
\label{S1p}
\left|I_{1}\right|\lesssim\Delta^{p+1},\quad\left|I_{2}\right|\lesssim\Delta^{2p},\quad\Delta\to0.
\end{equation}
Similarly 
\begin{equation*}
\begin{split}
S_{2}^{(p)}(u)&=\int_{-1/\alpha}^{-1/\alpha+\Delta}(\mathcal{K}_{\alpha}(x))^{p}\psi^{(p)}\bigl(u(\mathcal{K}_{\alpha}(x))\bigr)dx\\
&=\int_{-1/\alpha}^{-1/\alpha+\Delta}\frac{\left(\mathcal{K}_{\alpha}(x)\right)^{p}}{\mathcal{K}_{\alpha}'(x)}\psi^{(p)}\bigl(u(\mathcal{K}_{\alpha}(x))\bigr)d\mathcal{K}_{\alpha}(x)\\
&=\int_{0}^{(\alpha\Delta)^{1/\alpha}}y^{p+\alpha-1}\psi^{(p)}(uy)\,dy\\
&=u^{-p-\alpha}\int_{0}^{u(\alpha\Delta)^{1/\alpha}}z^{p+\alpha-1}\psi^{(p)}(z)\,dz.
\end{split}
\end{equation*} 
Further 
\begin{equation*}
\psi^{(p)}(z)=\mathrm{i}^p \int_{\mathbb{R}}x^p e^{\mathrm{i}uz}\nu(x)\,dx\lesssim \int|x|^{p}\nu(x)\,dx
\end{equation*}
and hence
\begin{equation}
\label{S2p}
\begin{split}
S_{2}^{(p)}(u/\Delta)&=(u/\Delta)^{-p-\alpha}\int_{0}^{u(\alpha\Delta)^{1/\alpha}\Delta^{-1}}z^{p+\alpha-1}\psi^{(p)}(z)\,dz\\
&\lesssim (u/\Delta)^{-p-\alpha}\int_{0}^{u(\alpha\Delta)^{1/\alpha}\Delta^{-1}}z^{p+\alpha-1}\left( \int|x|^{p}\nu(x)\,dx\right)dz\\
&\lesssim \Delta^{(p+\alpha)/\alpha}.
\end{split}
\end{equation}
Combining \ref{S1p} and \ref{S2p}, we get
\begin{equation*}
S_{\Delta}^{(p)}(u/\Delta)\lesssim \Delta^{p+1}.
\end{equation*}
Since $0<\alpha<1$, it follows that
\begin{equation*}
\lim_{\Delta\to0}\Delta^{-(p+\alpha)}S_{\Delta}^{(p)}(u/\Delta)=0.
\end{equation*}
\end{proof}
Next, we formulate and prove some auxiliary lemmas that we need to prove the main results.
In the sequel we assume that \(h_n,\Delta\to 0\) as \(n\to \infty\) and \(\lesssim\) stands for inequality up to a constant not depending on \(h_n,\Delta\) and \(n.\)
\begin{lem}
\label{eq:Phi-low}
We have for any \(h_n>0\) with $h_n^{3}\gtrsim\Delta,$
\begin{eqnarray*}
\inf_{|u|\leq h_n^{-1}}\bigl|\Phi_{\Delta X}(u)\bigr|\gtrsim 1.
\end{eqnarray*}
\label{infPhi}
\end{lem}
\begin{proof}
Recall that
\begin{equation*}
\Phi_{\Delta X}(u)=\exp\left[\frac{2\Delta}{1-\alpha}u^{-\frac{\alpha}{1-\alpha}}\int_{0}^{u}\psi(z)z^{\frac{2\alpha-1}{1-\alpha}}\,dz\right].
\end{equation*}
By using the Taylor expansion  we obtain for any $u,\ x\in\mathbb{R},$
\begin{eqnarray*}
 \bigl|e^{\mathrm{i}u x}-1-\mathrm{i}ux\mathbb{I}_{|x|\leq1}\bigr|\leq\left|\frac{x^{2}u^{2}}{2}\right|,
 \end{eqnarray*}
so that 
\begin{equation*}
|\psi(u)|\leq\frac{\sigma^{2}u^{2}}{2}+|\gamma u|+\frac{u^{2}}{2}\int_{\mathbb{R}}x^{2}\,\nu(x)\,dx\lesssim h_n^{-2}.
\end{equation*} 
Then the infimum of $\Phi_{\Delta X}(u)$ can under the condition $h_n^{3}\gtrsim\Delta$ be estimated by
\begin{equation*}
\begin{split}
\inf_{|u|\leq h_n^{-1}}\bigl|\Phi_{\Delta X}(u)\bigr| & \geq \exp \biggl(-\Delta\sup_{|u|\leq h_n^{-1}}\biggl|\frac{2}{1-\alpha}u^{-\frac{\alpha}{1-\alpha}}\int_{0}^{u}\psi(z)z^{\frac{2\alpha-1}{1-\alpha}}\,dz\biggr|\biggr)\\
 & \geq \exp \biggl(-\Delta\sup_{|u|\leq h_n^{-1}}\biggl|\frac{\sigma^{2}u^{2}}{2}+\gamma u\biggr|\biggr)\\
 &\gtrsim e^{-\Delta h_n^{-2}}\gtrsim e^{-h_n}\gtrsim 1.
 \end{split}
\end{equation*}
This completes the proof.
\end{proof}

\begin{lem}
Define $\rho(x)=x^{2}\nu(x)$, then $\left\Vert \rho\ast\left(h_n^{-1}W\left(\cdot/h_n\right)\right)-\rho\right\Vert _{\mathbb{R}}\lesssim h_n^{r}$,  where $\ast$ denotes the convolution.
\label{2TermRest}
\end{lem}
\begin{proof}
Using the change of variables, we may rewrite the term as
\begin{eqnarray*}
 \rho\left(h_n^{-1}W\left(\cdot/h_n\right)\right)-\rho(x)=\int_{\mathbb{R}}\left\{ \rho(x-yh_n)-\rho(x)\right\} W(y)\,dy.
\end{eqnarray*}
By the Tailor's expansion  we obtain  for any $x,y\in\mathbb{R},$
\begin{eqnarray*}
\rho(x-yh_n)-\rho(x)=\stackrel[i=1]{p-1}{\sum}\frac{\rho^{(l)}(x)}{l!}(-yh_n)^{l}+\frac{\rho^{(p)}(x-\theta yh_n)}{p!}(-yh_n)^{p},
\end{eqnarray*}
for some $\theta\in(0,1$). Since $\rho^{(p)}\left(r-p\right)$ $\alpha-$Hölder continuous,
we can assert that $H\coloneqq\underset{x,y\in\mathbb{R},x\neq y}{\sup}\frac{\left|\rho^{(p)}(x)-\rho^{(p)}(y)\right|}{|x-y|^{r-p}}<\infty.$

Furthermore, since $\int_{\mathbb{R}}y^{l}W(y)\,dy=0,\ for\ l=1,\ldots,p,$
 we conclude that for any $x\in R,$
\begin{equation*}
\begin{split}
\biggl|\int_{\mathbb{R}}\bigl\{\rho(x-yh_n)-\rho(x)\bigr\} W(y)dy\biggr|&=\biggl|\int_{\mathbb{R}}\bigl[\rho(x-yh_n)-\rho(x)-\stackrel[i=1]{p}{\sum}\frac{\rho^{l}(x)}{l!}(-yh_n)^{l}\bigr]W(y)\,dy\biggr|\\
&=\biggl|\int_{\mathbb{R}}\bigl[\frac{\rho^{(p)}(x-\theta yh_n)-\rho^{(p)}(x)}{p!}(-yh_n)^{p}\bigr]W(y)\,dy\biggr|\\
&=\biggl|\int_{\mathbb{R}}\bigl[\frac{\rho^{(p)}(x-\theta yh_n)-\rho^{(p)}(x)}{p!(-\theta yh_n)^{r-p}}(-yh_n)^{r}\theta^{r-p}\bigr]W(y)\,dy\biggr|\\
 &\leq\frac{Hh_n^{r}}{p!}\int_{\mathbb{R}}|y|^{r}|W(y)|\,dy.
 \end{split}
\end{equation*}
This completes the proof.
\end{proof}
\begin{lem}
\label{DifPhiPhi}
Suppose that $\int_{\mathbb{R}}|x|^{p}\nu(x)\,dx<\infty $ for some $p\geq1$ and let 
\begin{equation}
\label{Psi}
\Psi(u):=\frac{2}{1-\alpha}\int_{0}^{1}\psi(u y)y^{\frac{2\alpha-1}{1-\alpha}}\,dy,
\end{equation}
then we have \footnote{The supremum is found on the interval $I_{1}$ so that $\varphi_{W}(uh_n)$ supported in $\left[-1,1\right]$  }
\begin{enumerate}[label=(\roman*)] 
\item 
$\ \bigl\Vert \Psi'\bigr\Vert_{I_{h}}\lesssim \frac{1}{h_n},$
\item 
$\ \bigl\Vert\Psi^{(k)}\bigr\Vert_{I_{h}}\lesssim 1,\quad k\geq2.$
\end{enumerate}
\end{lem}
\begin{proof}
Under the assumption we have 
\begin{equation*}
\begin{split}
&\Psi'(u)=\frac{2}{1-\alpha}\int_{0}^{1}y\psi'(uy)y^{\frac{2\alpha-1}{1-\alpha}}\,dy\\
&\hspace{0.87cm}=\frac{2}{1-\alpha}\int_{0}^{1}\biggl(-\sigma^{2}uy^{2}+\mathrm{i}y\bigl(\gamma+\int_{\mathbb{R}}x(e^{\mathrm{i}uyx}-\mathbb{I}_{\{|x|\leq1\}})\nu(dx)\bigr)\biggr)y^{\frac{2\alpha-1}{1-\alpha}}\,dy,\\
&\Psi''(u)=\frac{2}{1-\alpha}\int_{0}^{1}y^{2}\psi''(uy)y^{\frac{2\alpha-1}{1-\alpha}}\,dy	=\int_{0}^{1}\biggl(-\sigma^{2}y^{2}-y^{2}\int_{\mathbb{R}}x^{2}e^{\mathrm{i}uyx}\nu(dx)\biggr)y{}^{\frac{2\alpha-1}{1-\alpha}}\,dy,\\
&\Psi^{(k)}(u)=\frac{2}{1-\alpha}\int_{0}^{1}y^{k}\psi^{(k)}(uy)y^{\frac{2\alpha-1}{1-\alpha}}\,dy
	=\int_{0}^{1}y^{k}\biggl[\int_{\mathbb{R}}x^{k}e^{\mathrm{i}uyx}\nu(dx)\biggr]y^{\frac{2\alpha-1}{1-\alpha}}\,dy,\,\, k>2.
\end{split}
\end{equation*}
and the assertion follows. 
\end{proof}
\begin{lem}
\label{Moments}
For $k\in \mathbb{N}$ we have
\begin{equation*}
\mathsf{E}\bigl[|\Delta X
|^{2k}\bigl]\lesssim \Delta,
\end{equation*}
where $(\Delta X)_t$ the increment process of the limiting L\'evy process $(X_t)$.
\end{lem}
\begin{proof}
Since $\Phi_{\Delta X}(u)=\exp\bigl[ \Delta \Psi(u)\bigr],$ we have
\begin{equation*}
\mathsf{E}\bigl[\Delta X^{2k}\bigr]=(-\mathrm{i})^{2k}\Phi_{\Delta X}^{(2k)}(0)=(-1)^{k}\frac{d^{2k}}{du^{2k}}\exp\bigl[\Delta\Psi(u)\bigr]\mid_{u=0}.
\end{equation*}
Note that
\begin{equation*}
\Phi_{\Delta X}^{(2k)}(u)=\Delta \Phi_{\Delta X}(u)\bigl(\Psi^{(2k)}(u)+a_{1}\Delta \Psi^{(2k-1)}(u)\Psi'(u)+\ldots+\Delta^{2k-1}(\Psi'(u))^{2k}\bigr)
\end{equation*}
where $a_{i}\in\mathbb{R},$ for $i=1,\ldots,2k-1$.
 This observation completes the proof.
\end{proof}
\begin{lem}
\label{dPhiest}
For $k=0,1,2,3$ we have
\begin{equation*}
\left\Vert \mathbb{D}^{(k)}(u)\right\Vert _{I_{h}}=\mathcal{O}_{P}\bigl(n^{-1/2}\Delta^{(k \wedge 1 )/2} \log h_n^{-1}\bigr), 
\end{equation*}
where $\mathbb{D}(u)=\widehat{\Phi}_{\Delta X}(u)-\Phi_{\Delta X}(u).$
\end{lem}
\begin{proof}
To prove this statement we use Lemma \ref{Moments} together with proof of the Theorems 1 (see \cite{sato1999levy}), which shows that for the weight function $\omega(u)=\left(\log\left(e+\left|u\right|\right)\right)^{-1}$ under Assumption \ref{assumption1}, we obtain 
\begin{equation*}
C_{k}:=\underset{n}{\sup}\,\mathsf{E}\left[\bigl\Vert \sqrt{n}\Delta^{-(k \wedge 1)/2}\mathbb{D}^{(k)}(u)\omega(u)\bigr\Vert_{\mathbb{R}}\right]<\infty
\end{equation*}
for $k=0,1,2,3$.
Furthermore,
\begin{equation*}
\bigl\Vert \sqrt{n}\Delta^{-(k \wedge 1 )/2}\mathbb{D}^{(k)}(u)\omega(u)\bigr\Vert _{\mathbb{R}}\geq \sqrt{n}\Delta^{-(k \wedge 1 )/2}\bigl\Vert \mathbb{D}^{(k)}(u)\bigr\Vert_{I_{h}}\underset{|u|\leq h_n^{-1}}{\inf}\omega(u)
\end{equation*}
Since
\begin{equation*}
\begin{split}
\underset{|u|\leq h_n^{-1}}{\inf}\omega(u)&=\underset{|u|\leq h_n^{-1}}{\inf}\left(\log(e+|u|)\right)^{-1}=\bigl(\underset{|u|\leq h_n^{-1}}{\sup}\log(e+|u|)\bigr)^{-1}\\
&=\log\left(e+h_n^{-1}\right)^{-1}
\end{split}
\end{equation*}
we conclude that 
\begin{equation*}
\mathsf{E}\left[\bigl\Vert \mathbb{D}^{(k)}(u)\bigr\Vert_{I_{h}}\right]\leq\frac{C_{k}\Delta^{(k \wedge 1 )/2}}{\sqrt{n}\underset{|u|\leq h_n^{-1}}{\inf}\omega(u)}\lesssim n^{-1/2}\Delta^{(k \wedge 1 )/2}\log h_n^{-1}.
\end{equation*}
This completes the proof.
\end{proof}
Note that the Lemma~\ref{infPhi} and Lemma~\ref{dPhiest}  imply 
\begin{eqnarray*}
\inf_{|u|\leq h_n^{-1}}\bigl|\widehat{\Phi}_{\Delta X}(u)\bigr|\geq\inf_{|u|\leq h_n^{-1}}\bigl|\Phi_{\Delta X}(u)\bigr|-o_{p}(1)\gtrsim 1-o_{p}(1).
\end{eqnarray*}
\begin{lem}
\label{2term}
Let $R_n(x)$ be the form
\begin{equation*}
\begin{split}
R_n(x)&=-\frac{1}{2\pi\Delta}\int_{\mathbb{R}}e^{-\mathrm{i}u x}\bigl[Q_{0}(u)\mathbb{D}(u)+Q_{1}(u)\mathbb{D}'(u)\bigr.\\
&\bigl.\hspace{2cm}+Q_{2}(u)\mathbb{D}''(u)+Q_{3}(u)\mathbb{D}'''(u)\bigr]\varphi_{W}(uh_n)\, du,
\end{split}
\end{equation*}
where $\mathbb{D}(u):=\widehat{\Phi}_{\Delta X}(u)-\Phi_{\Delta X}(u)$  and
\begin{equation*}
\begin{split}
Q_{0}(u)&=\frac{\Delta}{\Phi_{\Delta X}(u)}\left(\frac{2-\alpha}{2}\left(\Delta(\Psi'(u))^{2}-\Psi''(u)\right)\right.\\
&\left.\hspace{2cm}-u\frac{1-\alpha}{2}\left(\Psi'''(u)-3\Delta\Psi''(u)\Psi'(u)+\Delta^2(\Psi'(u))^3\right)\right),\\
Q_{1}(u)&=\frac{\Delta}{\Phi_{\Delta X}(u)}\left(3u\frac{1-\alpha}{2}\left(\Delta(\Psi'(u))^{2}-\Psi''(u)\right)-(2-\alpha)\Psi'(u)\right),\\
Q_{2}(u)	&=\frac{1}{\Phi_{\Delta X}(u)}\left(\frac{2-\alpha}{2}-3u\Delta\Psi'(u)\frac{1-\alpha}{2}\right),\\
Q_{3}(u)&=\frac{1}{\Phi_{\Delta X}(u)}\left(u\frac{1-\alpha}{2}\right).
\end{split}
\end{equation*}
Then applies
\begin{equation*}
\Vert R_n(x)\Vert_{\mathbb{R}}=\mathcal{O}\left(\Delta^{-1/2} h_n^{-1}n^{-1/2}\log h_n^{-1}\right).
\end{equation*}
\end{lem}
\begin{proof}
Let us first consider the integral $\,I_0(x):=\int_{\mathbb{R}}e^{-\mathrm{i}u x} Q_{0}(u)\mathbb{D}(u)\varphi_{W}(uh_n)\, du$. 
\begin{equation*}
\begin{split}
&I_0(x)=\int_{\mathbb{R}}e^{-\mathrm{i}u x} \mathbb{D}(u) \frac{\Delta}{\Phi_{\Delta X}(u)}\left(\frac{2-\alpha}{2}\left(\Delta(\Psi'(u))^{2}-\Psi''(u)\right)\right.\\
&\left.\hspace{2cm}-u\frac{1-\alpha}{2}\left(\Psi'''(u)-3\Delta\Psi''(u)\Psi'(u)+\Delta^2(\Psi'(u))^3\right)\right)\varphi_{W}(uh_n)\, du
\end{split}
\end{equation*}
According to the lemmas \ref{dPhiest}, \ref{DifPhiPhi} and \ref{infPhi} we get
\begin{equation*}
\begin{split}
\Vert I_0(x)\Vert_{\mathbb{R}}&\lesssim \Delta \left\Vert \mathbb{D}(u)\right\Vert _{I_{h}}\left(\Delta h_n^{-2}+1+h_n^{-1}(1+\Delta h_n^{-1}+\Delta^2 h_n^{-3})\right) \\
&\lesssim \Delta h_n^{-1}n^{-1/2} \log h_n^{-1}.
\end{split}
\end{equation*}
Analogous to $\Vert I_0(x)\Vert_{\mathbb{R}}$ applies to the integrals $I_i(x):=\int_{\mathbb{R}}e^{-\mathrm{i}u x} Q_{i}(u)\mathbb{D}^{\left(i\right)}(u)\varphi_{W}(uh_n)\, du$ for $i=1,2,3$:
\begin{equation*}
\begin{split}
\Vert I_1(x)\Vert_{\mathbb{R}}&\lesssim \Delta \left\Vert \mathbb{D}'(u)\right\Vert _{I_{h}}\left( h_n^{-1}(\Delta h_n^{-2}+1)+h_n^{-1}\right) \\
&\lesssim \Delta^{3/2} h_n^{-1}n^{-1/2}\log h_n^{-1},\\
\Vert I_2(x)\Vert_{\mathbb{R}}&\lesssim  \left\Vert \mathbb{D}''(u)\right\Vert _{I_{h}}\left(1+ \Delta h_n^{-2}\right) \\
&\lesssim \Delta^{1/2} n^{-1/2}\log h_n^{-1},\\
\Vert I_3(x)\Vert_{\mathbb{R}}&\lesssim  \left\Vert \mathbb{D}'''(u)\right\Vert _{I_{h}}h_n^{-1}\\
&\lesssim \Delta^{1/2} h_n^{-1}n^{-1/2}\log h_n^{-1}.
\end{split}
\end{equation*}
Thus the claim of the lemma holds such that
\begin{equation*}
\Vert R_n(x)\Vert_{\mathbb{R}}=\mathcal{O}\left(\Delta^{-1/2} h_n^{-1}n^{-1/2}\log h_n^{-1}\right).
\end{equation*}
\end{proof}
\begin{lem}
\label{InfPDelta}

Let $P_{\Delta}$ denote the distribution of the r.v. $X_{\Delta}-X_0$.  The measures $y^{2m}P_{\Delta}(dy)$  for $m=1,2,3$ have Lebesgue densities $g_{\Delta}^{2m}$.  For any compact set in $\mathbb{R}\setminus\{ 0\}$ and any $\varepsilon_{0}>0$, let $I^{\varepsilon_{0}}=\{ x\in\mathbb{R}:\,\,d(x,I)\leq\varepsilon_{0}\}$, where $d(x,I)=\inf_{y\in I}|x-y|$. 

We have
\begin{enumerate}[label=(\roman*)] 
\item
\begin{equation*}
\inf_{y\in I^{\varepsilon_{0}}}g_{\Delta}^{2m}(y)=\inf_{y\in I^{\varepsilon_{0}}}\left(y^{2m}P_{\Delta}\right)(y)\gtrsim\Delta,\,\text{ for }, \, m=1,2,3 \quad \Delta\to 0,
\end{equation*}
\item
\begin{equation*}
\Vert g_{\Delta}^{2}\Vert_{\mathbb{R}}\lesssim \Delta^{1/2}\, \,\mathrm{and}\,\,
\Vert g_{\Delta}^{2m}\Vert_{\mathbb{R}}\lesssim \Delta,\quad m=2,3 \quad \Delta\to 0,
\end{equation*}
\end{enumerate}
for some sufficiently small $\varepsilon_{0}>0$ such that $0\notin I^{\varepsilon_{0}}$. 
 \label{ass2}
\end{lem}
\begin{proof}  
Note that
\begin{equation*}
\begin{split}
&\psi'(uy)=-u\sigma^{2}y^{2}+\mathrm{i}y\bigl(\gamma+\int_{\mathbb{R}}x\bigl(e^{\mathrm{i}uyx}-\mathbb{I}_{[-1,1]}(x)\bigr)\nu(dx)\bigr)\\
&\psi''(uy) =-\sigma^{2}y^{2}-\int_{\mathbb{R}}e^{\mathrm{i}uyx}(yx)^{2}\nu(x)\,dx\\
 &\hspace{1.1cm} =-\int_{\mathbb{R}}y^{-1}e^{\mathrm{i}ut}t^{2}\nu_{\sigma}(t/y)\,dt\\
&\psi^{(k)}(uy)  =\int_{\mathbb{R}}e^{\mathrm{i}uyx}(\mathrm{i}yx)^{k}\nu(x)dx=\int_{\mathbb{R}}y^{-1}e^{\mathrm{i}ut}t^{k}\nu(t/y)\,dt,
\end{split}
\end{equation*}
where $\nu_{\sigma}=\sigma^{2}y^{2}\delta_{0}+t^{2}\nu\ensuremath{,\,}t=yx.$ It also follows that
\begin{equation*}
\begin{split}
&\Psi'(u)= \frac{2}{1-\alpha}\int_{0}^{1}\bigl(-u\sigma^{2}y^{2}+\mathrm{i}y\bigl(\gamma+\int_{\mathbb{R}}x\bigl(e^{\mathrm{i}uyx}-\mathbb{I}_{[-1,1]}(x)\bigr)\nu(dx)\bigr)y^{\frac{2\alpha-1}{1-\alpha}}\bigr)\,dy\\
 &\hspace{0.9cm} =\frac{2}{1-\alpha}\int_{0}^{1}\biggl(\int_{\mathbb{R}}e^{\mathrm{i}ut}\bigl(-\sigma^{2}\delta_{1}uy^{2}+\delta_{2}\mathrm{i}\gamma y+t\nu(t/y)\bigr)\,dt\biggr)y^{\frac{3\alpha-2}{1-\alpha}}\,dy\\
 & \hspace{0.9cm}=\frac{2}{1-\alpha}\int_{\mathbb{R}}e^{\mathrm{i}ut}\biggl(\int_{0}^{1}\bigl(-\sigma^{2}\delta_{1}uy^{2}+\delta_{2}\mathrm{i}\gamma y+t\nu(t/y)\bigr)y^{\frac{3\alpha-2}{1-\alpha}}\,dy\biggr)\,dt,\\
&\Psi''(u)=\frac{2}{1-\alpha}\int_{0}^{1}\psi''(uy)y^{\frac{2\alpha-1}{1-\alpha}}\,dy=\frac{2}{1-\alpha}\int_{0}^{1}\biggl(\int_{\mathbb{R}}e^{\mathrm{i}ut}t^{2}\nu_{\sigma}(t/y)\,dt\biggr)y^{\frac{3\alpha-2}{1-\alpha}}\,dy\\
 &\hspace{0.9cm} =\frac{2}{1-\alpha}\int_{\mathbb{R}}e^{\mathrm{i}ut}t^{2}\int_{0}^{1}y^{\frac{3\alpha-2}{1-\alpha}}\nu_{\sigma}(t/y)\,dy\,dt,\\
& \Psi^{(k)}(u)  =\frac{2}{1-\alpha}\int_{0}^{1}\psi^{(k)}(uy)y^{\frac{2\alpha-1}{1-\alpha}}\,dy=\frac{2}{1-\alpha}\int_{\mathbb{R}}e^{\mathrm{i}ut}t^{k}\int_{0}^{1}y^{\frac{3\alpha-2}{1-\alpha}}\nu(t/y)\,dy\,dt.
\end{split}
\end{equation*}
From infinite divisibility of the process $(\Delta X)_t$ it follows that $\Phi_{\Delta X}(u)=\bigl(\Phi_{\Delta/2}(u)\bigr)^{2}$. Then 
\begin{align*}
\Phi''_{\Delta X}(u) & =\Phi_{\Delta X}(u)\bigl(\Delta\Psi''(u)+(\Delta\Psi'(u))^{2}\bigr)\\
 & =\Delta\Psi''(u)\Phi_{\Delta X}(u)+4\bigl(\Delta/2\Psi'(u)\Phi_{\Delta/2}(u)\bigr)^{2}\\
 & =\Delta\Psi''(u)\Phi_{\Delta X}(u)+4\bigl(\Phi'_{\Delta/2}(u)\bigr)^{2}.
\end{align*}
Furthermore
\begin{equation*}
\varint_{\mathbb{R}}e^{\mathrm{i}ux}x^{2}P_{\Delta}(dx)  =\frac{2}{1-\alpha}\Delta\int_{\mathbb{R}}e^{\mathrm{i}ut}t^{2}\int_{0}^{1}y^{\frac{3\alpha-2}{1-\alpha}}\nu_{\sigma}(t/y)\,dy\,dt\varint_{\mathbb{R}}e^{\mathrm{i}ut}P_{\Delta}(dt)+4\bigl(\varint_{\mathbb{R}}e^{\mathrm{i}ut}tP_{\Delta/2}(dt)\bigr)^{2}.
\end{equation*}
Hence
\begin{equation}
\label{Faltung}
\begin{split}
x^{2}P_{\Delta} & =\frac{2}{1-\alpha}\Delta\biggl(t^{2}\int_{0}^{1}y^{^{\frac{3\alpha-2}{1-\alpha}}}\nu_{\sigma}(t/y)\,dy\biggr)\ast P_{\Delta}+4(xP_{\Delta/2})\ast(xP_{\Delta/2}),\\
x^{3}P_{\Delta} & =\frac{2}{1-\alpha}\Delta\biggl(\biggl(t^{3}\int_{0}^{1}y^{^{\frac{3\alpha-2}{1-\alpha}}}\nu(t/y)\,dy\biggr)\ast P_{\Delta}+\biggr.\\
&\biggl.+\biggl(t^{2}\int_{0}^{1}y^{^{\frac{3\alpha-2}{1-\alpha}}}\nu_{\sigma}(t/y)\,dy\biggr)\ast(xP_{\Delta})\biggr)+8(xP_{\Delta/2})\ast(x^{2}P_{\Delta/2}),\\
x^{4}P_{\Delta} & =\frac{2}{1-\alpha}\Delta\biggl(\biggl(t^{4}\int_{0}^{1}y^{\frac{3\alpha-2}{1-\alpha}}\nu(t/y)\,dy\biggr)\ast P_{\Delta}\biggr.\\
&\biggl.+2\biggl(t^{3}\int_{0}^{1}y^{\frac{3\alpha-2}{1-\alpha}}\nu(t/y)\,dy\biggr)\ast(xP_{\Delta})+\biggl(t^{2}\int_{0}^{1}y^{\frac{3\alpha-2}{1-\alpha}}\nu_{\sigma}(t/y)\,dy\biggr)\ast(x^{2}P_{\Delta})\biggr)\\
&+8(x^{2}P_{\Delta/2})\ast(x^{2}P_{\Delta/2})+8(x^{3}P_{\Delta/2})\ast(xP_{\Delta/2}),\\
x^{6}P_{\Delta} & =\frac{2}{1-\alpha}\Delta\biggl(\biggl(t^{6}\int_{0}^{1}y^{\frac{3\alpha-2}{1-\alpha}}\nu(t/y)\,dy\biggr)\ast P_{\Delta}+4\biggl(t^{5}\int_{0}^{1}y^{\frac{3\alpha-2}{1-\alpha}}\nu(t/y)\,dy\biggr)\ast(xP_{\Delta})\biggr.\\
 & \biggl.+6\biggl(t^{4}\int_{0}^{1}y^{\frac{3\alpha-2}{1-\alpha}}\nu(t/y)\,dy\biggr)\ast(x^{2}P_{\Delta})+4\biggl(t^{3}\int_{0}^{1}y^{\frac{3\alpha-2}{1-\alpha}}\nu(t/y)\,dy\biggr)\ast(x^{3}P_{\Delta})\biggr.\\
 & \biggl.+\biggl(t^{2}\int_{0}^{1}y^{\frac{3\alpha-2}{1-\alpha}}\nu_{\sigma}(t/y)\,dy\biggr)\ast(x^{4}P_{\Delta})\biggr)+32(x^{4}P_{\Delta/2})\ast(x^{2}P_{\Delta/2})\\
 & +24(x^{3}P_{\Delta/2})\ast(x^{3}P_{\Delta/2})+8(x^{5}P_{\Delta/2})\ast(xP_{\Delta/2}).
\end{split}
\end{equation}

%Furthermore let \textcolor{red}{$|\nu(x)|\geq|x|^{-\beta}$}, then we have
%\begin{equation*}
%\begin{split}
%t^{k}\int_{0}^{1}y^{\frac{3\alpha-2}{1-\alpha}}\nu(t/y)\,dy & \geq t^{\frac{2\alpha-1}{1-\alpha}+k}\int_{t}^{\infty}x^{-\frac{\alpha}{1-\alpha}}\nu(x)\,dx\\
 %& \geq t^{\frac{2\alpha-1}{1-\alpha}+k}\int_{t}^{\infty}|x|^{-\left(\frac{\alpha}{1-\alpha}+\beta\right)}\,dx\\
 %& =\biggl(\frac{2\alpha-1}{1-\alpha}+\beta\biggr)^{-1}t^{k-\beta}
%\end{split}
%\end{equation*}
%for some \textcolor{red}{$\beta>1-\frac{\alpha}{1-\alpha}$}. Finally, we have 
%\begin{equation*}
%\begin{split}
%\underset{t\in I^{\varepsilon_{1}}}{\inf}\Bigl(t^{k}\int_{0}^{1}y^{\frac{3\alpha-2}{1-\alpha}}\nu(t/y)\,dy\Bigr)& \geq\biggl(\frac{2\alpha-1}{1-\alpha}+\beta\biggr)^{-1}\underset{t\in I^{\varepsilon_{1}}}{\inf}t^{k-\beta}\\
 %& =\biggl(\frac{2\alpha-1}{1-\alpha}+\beta\biggr)^{-1}\underset{x\in I^{\varepsilon_{0}}}{\inf}(x)^{k-\beta}=C_{inf}\gtrsim1.
%\end{split}
%\end{equation*}
 For any $\varepsilon>0,$ it holds due to the Markov inequality
\begin{equation*}
P_{\Delta}\bigl([-\varepsilon,\varepsilon]^{c}\bigr)=P\bigl(|\Delta X|>\varepsilon\bigr)\leq\varepsilon^{-2}\mathsf{E}\bigl[\Delta X^{2}\bigr]\lesssim\Delta.
\end{equation*}
Then it follows $P_{\Delta}\bigl([-\varepsilon,\varepsilon]\bigr)=1-\mathcal{O}(\Delta)$ and since $\underset{t\in I^{\varepsilon_{1}}}{\inf}\left(t^{k}\int_{0}^{1}y^{\frac{3\alpha-2}{1-\alpha}}\nu(t/y)\,dy\right) \gtrsim1$ we have
\begin{equation*}
\begin{split}
&\underset{t\in I^{\varepsilon_{1}}}{\inf}\left(\left(t^{k}\int_{0}^{1}y^{\frac{3\alpha-2}{1-\alpha}}\nu(t/y)\,dy\right)\ast P_{\Delta}\right) \\
&\hspace{2cm} \geq P_{\Delta}\bigl[-\varepsilon_{1}/2,\varepsilon_{1}/2\bigr]\underset{t\in I^{\varepsilon_{1}}}{\inf}\left(t^{k}\int_{0}^{1}y^{\frac{3\alpha-2}{1-\alpha}}\nu(t/y)\,dy\right) \gtrsim1.
\end{split} 
\end{equation*} 
Then the first claim follows from \eqref{Faltung},
\begin{equation*}
\inf_{y\in I^{\varepsilon_{0}}}g_{\Delta}^{2m}(y)=\inf_{y\in I^{\varepsilon_{0}}}\bigl(y^{2m}P_{\Delta}\bigr)(y)\gtrsim\Delta,\quad m=1,2,3.
\end{equation*} 
\item
Furthermore due   \eqref{Faltung} and  
\begin{equation*}
\bigl\Vert y P_{\Delta/2}\bigr\Vert_{L^1}=\mathsf{E}\bigl[X_{\Delta/2}\bigr]\leq\sqrt{\mathsf{E}\bigl[X^2_{\Delta/2}\bigr]}\lesssim\Delta^{1/2},
\end{equation*}
we derive
\begin{equation*}
\begin{split}
\bigl\Vert y^2 P_{\Delta}\bigr\Vert_{\mathbb{R}}&\lesssim\Delta\bigl\Vert y^2\nu\bigr\Vert_{\mathbb{R}}P_{\Delta}(\mathbb{R
})+\bigl\Vert y P_{\Delta/2}\bigr\Vert_{\mathbb{R}}\bigl\Vert y P_{\Delta/2}\bigr\Vert_{L^1}\lesssim\Delta+\Delta^{1/2}\lesssim\Delta^{1/2},\\
\bigl\Vert y^3 P_{\Delta}\bigr\Vert_{\mathbb{R}}&\lesssim\Delta\bigl\{ \bigl\Vert y^3\nu\bigr\Vert_{\mathbb{R}}P_{\Delta}(\mathbb{R
})+\bigl\Vert y P_{\Delta}\bigr\Vert_{\mathbb{R}}\nu_{\sigma}(\mathbb{R
})\bigr\} \\
&\hspace{3.5cm}+\bigl\Vert y P_{\Delta/2}\bigr\Vert_{\mathbb{R}}\bigl\Vert y^2 P_{\Delta/2}\bigr\Vert_{L^1}\lesssim\Delta,\\
\bigl\Vert y^4 P_{\Delta}\bigr\Vert_{\mathbb{R}}&\lesssim\Delta\bigl\{ \bigl\Vert y^4\nu\bigr\Vert_{\mathbb{R}}P_{\Delta}(\mathbb{R
})+\bigl\Vert y^3\nu\bigr\Vert_{\mathbb{R}}\bigl\Vert y P_{\Delta}\bigr\Vert_{L^1}+\Vert y^2 P_{\Delta}\Vert_{\mathbb{R}}\nu_{\sigma}\left(\mathbb{R
}\right)\bigr\} \\
&+\bigl\Vert y^3 P_{\Delta/2}\bigr\Vert_{\mathbb{R}}\bigl\Vert y P_{\Delta/2}\bigr\Vert_{L^1}+\bigl\Vert y^2 P_{\Delta/2}\bigr\Vert_{\mathbb{R}}\bigl\Vert y^2 P_{\Delta/2}\bigr\Vert_{L^1}\lesssim\Delta,\\
\bigl\Vert y^5 P_{\Delta}\bigr\Vert_{\mathbb{R}}&\lesssim\Delta\bigl\{ \bigl\Vert y^5\nu\bigr\Vert_{\mathbb{R}}P_{\Delta}(\mathbb{R
})+\bigl\Vert y^4\nu\bigr\Vert_{\mathbb{R}}\bigl\Vert y P_{\Delta}\bigr\Vert_{L^1}+\bigl\Vert y^3\nu\bigr\Vert_{\mathbb{R}}\bigl\Vert y^2 P_{\Delta}\bigr\Vert_{L^1}\bigr.\\
&\bigl.\,+\bigl\Vert y^3 P_{\Delta}\bigr\Vert_{\mathbb{R}}\nu_{\sigma}(\mathbb{R
})\bigr\}+\bigl\Vert y^4 P_{\Delta/2}\bigr\Vert_{\mathbb{R}}\bigl\Vert y P_{\Delta/2}\bigr\Vert_{L^1}\\
&+\bigl\Vert y^3 P_{\Delta/2}\bigr\Vert_{\mathbb{R}}\bigl\Vert y^2 P_{\Delta/2}\bigr\Vert_{L^1}+\bigl\Vert y^2 P_{\Delta/2}\bigr\Vert_{\mathbb{R}}\bigl\Vert y^3 P_{\Delta/2}\bigr\Vert_{L^1}\lesssim\Delta,\\
\bigl\Vert y^6 P_{\Delta}\bigr\Vert_{\mathbb{R}}&\lesssim\Delta\bigl\{ \bigl\Vert y^6\nu\bigr\Vert_{\mathbb{R}}P_{\Delta}(\mathbb{R
})+\bigl\Vert y^5\nu\bigr\Vert_{\mathbb{R}}\bigl\Vert y P_{\Delta}\bigr\Vert_{L^1}+\bigl\Vert y^4\nu\bigr\Vert_{\mathbb{R}}\bigl\Vert y^2 P_{\Delta}\bigr\Vert_{L^1}\bigr.\\
&\bigl.\,+\bigl\Vert y^4 P_{\Delta}\bigr\Vert_{\mathbb{R}}\nu_{\sigma}(\mathbb{R
})+\bigl\Vert y^3\nu\bigr\Vert_{\mathbb{R}}\bigl\Vert y^3 P_{\Delta}\bigr\Vert_{L^1}\bigr\}+\bigl\Vert y^5 P_{\Delta/2}\bigr\Vert_{\mathbb{R}}\bigl\Vert y P_{\Delta/2}\bigr\Vert_{L^1}\\
&+\bigl\Vert y^4 P_{\Delta/2}\bigr\Vert_{\mathbb{R}}\bigl\Vert y^2 P_{\Delta/2}\bigr\Vert_{L^1}+\bigl\Vert y^3 P_{\Delta/2}\bigr\Vert_{\mathbb{R}}\bigl\Vert y^3 P_{\Delta/2}\bigr\Vert_{L^1}\\
&\hspace{2cm}+\bigl\Vert y^2 P_{\Delta/2}\bigr\Vert_{\mathbb{R}}\bigl\Vert y^4 P_{\Delta/2}\bigr\Vert_{L^1}\lesssim\Delta.
\end{split}
\end{equation*}
This completes the proof.
\end{proof}
\begin{lem}
For the variance $s^{2}(x)$ defined by \eqref{VarRn}, we have the following estimate:
\begin{eqnarray*} 
\inf_{x\in I}s^{2}(x)\gtrsim\Delta h_n^{-3},
\end{eqnarray*} 
  for sufficiently large $n$.
\label{lemvar}
\end{lem}
\begin{proof} 
We have
\begin{equation*}
\begin{split}
s^{2}(x)&=\stackrel[m=0]{3}{\sum}\mathrm{Var}\bigl[(\Delta X)_{1}^{m}K_{m,n}\bigl(x-(\Delta X)_{1}\bigr)\bigr]\\
&=\stackrel[m=0]{3}{\sum}\bigl(\mathsf{E}\bigl[(\Delta X)_{1}^{2m}K_{m,n}^2\bigl(x-(\Delta X)_{1}\bigr)\bigr]-\bigl(\mathsf{E}\bigl[(\Delta X)_{1}^{m}K_{m,n}\bigl(x-(\Delta X)_{1}\bigr)\bigr]\bigr)^2\bigr).
\end{split}
\end{equation*}
In order to determine the  infimum of the variance, we compute the supremum of the expected value $\mathsf{E}\bigl[(\Delta X)_{1}^{m}K_{m,n}\bigl(x-(\Delta X)_{1}\bigr)\bigr]$ and the infimum  of $\mathsf{E}\bigl[(\Delta X)_{1}^{2m}K_{m,n}^2\bigl(x-(\Delta X)_{1}\bigr)\bigr]$ for $m=0,1,2,3.$ Note that 
 \begin{equation*}
 \mathsf{E}\bigl[(\Delta X)_{1}^{m}K_{m,n}\bigl(x-(\Delta X)_{1}\bigr)\bigr]=\int_{\mathbb{R}}y^{m}K_{m,n}(x-y)P_{\Delta}\,(dy).
 \end{equation*}
 Further we get
 \begin{align*}
 \int_{\mathbb{R}}K_{0,n}(x-y)P_{\Delta}(dy)	&=\int_{\mathbb{R}}\biggl[\int e^{-\mathrm{i}u(x-y)}Q_{0}(u)\varphi_{W}(uh_{n})\,du\biggr]P_{\Delta}(dy)\\
	&=\int_{\mathbb{R}}\biggl[\int_{\mathbb{R}}e^{-\mathrm{i}u(x-y)}P_{\Delta}(dy)\biggr]Q_{0}(u)\varphi_{W}(uh_{n})\,du\\
	&=\int_{\mathbb{R}}e^{-\mathrm{i}ux}\Phi_{\Delta X}(u)Q_{0}(u)\varphi_{W}(uh_{n})\,du\\
	&=\Delta\int_{-1/h_n}^{1/h_n}e^{-\mathrm{i}ux}\biggl(\frac{2-\alpha}{2}\bigl(\Delta(\Psi'(u))^{2}-\Psi''(u)\bigr)\bigr.\\
	&\bigl.-u\frac{1-\alpha}{2}(\Psi'''(u)-3\Delta\Psi''(u)\Psi'(u)-\Delta^{2}(\Psi'(u))^{3})\biggr)du.
 \end{align*}
 Furthermore 
 \begin{align}
 \label{Int1}
 \int_{-1/h_n}^{1/h_n}e^{-\mathrm{i} ux}u\Psi'''(u)\,du&\leq\int_{-1/h_n}^{1/h_n}e^{-\mathrm{i}ux}u\,du=-2\mathrm{i}\int_{0}^{1/h_n}u\sin (ux)\,du\\
 \nonumber
 &=\frac{2i}{xh_n}\cos(x/h_n)+\frac{1}{x^{2}}\sin(x/h_n)\lesssim h_n^{-1}
 \end{align}
 for \(x\in I.\)
Analogously we get
 \begin{equation}
 \label{Int2}
 \begin{split}
& \int_{-1/h_n}^{1/h_n}e^{-\mathrm{i}ux}u\Psi''(u)\Psi'(u)\,du\leq h_n^{-1}\int_{-1/h_n}^{1/h_n}e^{-iux}u\,du\lesssim h_n^{-2},\\
&\int_{-1/h_n}^{1/h_n}e^{-\mathrm{i}ux}u(\Psi'(u))^{3}\,du\leq h^{-3}\int_{-1/h_n}^{1/h_n}e^{-\mathrm{i}ux}u\,du\lesssim h_n^{-4},\\
&\int_{-1/h_n}^{1/h_n}e^{-iux}(\Psi'(u))^{2}\,du\leq h_n^{-2}\int_{-1/h_n}^{1/h_n}e^{-\mathrm{i}ux}\,du\lesssim h_n^{-2},\\
&\int_{-1/h_n}^{1/h_n}e^{-\mathrm{i}ux}\Psi''(u)\,du\leq\int_{-1/h_n}^{1/h_n}e^{-\mathrm{i}ux}\,du\lesssim 1.
 \end{split}
 \end{equation}
 Then due to \eqref{Int1} and \eqref{Int2} we get
 \begin{equation}
 \label{SupK0}
 \int_{\mathbb{R}}K_{0,n}(x-y)P_{\Delta}(dy) \lesssim\Delta(\Delta h_n^{-2}+1+h_n^{-1}+\Delta^{2}h_n^{-4}) \lesssim\Delta h_n^{-1}.
 \end{equation}
Analogously 
\begin{equation}
\label{SupyK1}
\begin{split}
\int_{\mathbb{R}}yK_{1,n}(x-y)P_{\Delta}(dy)	&=\int_{\mathbb{R}}y\biggl[\int e^{-\mathrm{i}u(x-y)}Q_{1}(u)\varphi_{W}(uh_{n})\,du\biggr]P_{\Delta}(dy)\\
	&=\int_{\mathbb{R}}\biggl[\int_{\mathbb{R}}ye^{-\mathrm{i}u(x-y)}P_{\Delta}(dy)\biggr]Q_{1}(u)\varphi_{W}(uh_{n})\,du\\
	&=\int_{\mathbb{R}}e^{-\mathrm{i}ux}\Phi'_{\Delta X}(u)Q_{1}(u)\varphi_{W}(uh_{n})\,du\\
	&=\Delta^{2}\int^{1/h_n}_{-1/h_n} e^{-\mathrm{i}ux}\Psi'(u)\biggl(3u\frac{1-\alpha}{2}\bigl(\Delta(\Psi'(u))^{2}+\Psi''(u)\bigr)\biggr.\\
	&\biggl.\hspace{4cm}-(2-\alpha)\Psi'(u)\biggr)\, du\\
	&\lesssim\Delta^{2}(\Delta h_n^{-4}+h_n^{-2}+h_n^{-2}) \lesssim\Delta,
\end{split}
\end{equation}
\begin{equation}
\label{SupyK2}
\begin{split}
\int_{\mathbb{R}}y^{2}K_{2,n}(x-y)P_{\Delta}(dy)&=\int_{\mathbb{R}}y^{2}\biggl[\int e^{-\mathrm{i}u(x-y)}Q_{2}(u)\varphi_{W}(uh_{n})\,du\biggr]P_{\Delta}(dy)\\
	&=\int_{\mathbb{R}}\left[\int_{\mathbb{R}}y^{2}e^{-\mathrm{i}u(x-y)}P_{\Delta}(dy)\right]Q_{2}(u)\varphi_{W}(uh_{n})\,du\\
	&=\int_{\mathbb{R}}e^{-\mathrm{i}ux}\Phi''_{\Delta X}(u)Q_{2}(u)\varphi_{W}(uh_{n})\,du\\
	&=\Delta\int^{1/h_n}_{-1/h_n}e^{-\mathrm{i}ux}\bigl(\Psi''(u)+\Delta(\Psi'(u))^{2}\bigr)\\
	&\hspace{2cm}\times\biggl(\frac{2-\alpha}{2}-3u\Delta\Psi'(u)\frac{1-\alpha}{2}\biggr)\,du\\
	&\lesssim\Delta(1+\Delta h_n^{-2}+\Delta h_n^{-2}+\Delta^{2}h_n^{-4})\lesssim\Delta,
\end{split}	
\end{equation}
 \begin{equation}
 \label{SupyK3}
 \begin{split}
 \int_{\mathbb{R}}y^{3}K_{3,n}(x-y)P_{\Delta}(dy)&=\int_{\mathbb{R}}y^{3}\biggl[\int e^{-\mathrm{i}u(x-y)}Q_{3}(u)\varphi_{W}(uh_{n})\,du\biggr]P_{\Delta}(dy)\\
	&=\int_{\mathbb{R}}\biggl[\int_{\mathbb{R}}y^{3}e^{-\mathrm{i}u(x-y)}P_{\Delta}(dy)\biggr]Q_{3}(u)\varphi_{W}(uh_{n})\,du\\
	&=\int_{\mathbb{R}}e^{-\mathrm{i}ux}\Phi'''_{\Delta X}(u)Q_{3}(u)\varphi_{W}(uh_{n})\,du\\
		&=\Delta\int^{1/h_n}_{-1/h_n}e^{-\mathrm{i}ux}\biggl(u\frac{1-\alpha}{2}(\Psi'''(u)+3\Delta\Psi''(u)\Psi'(u)\biggr.\\
	&\biggl.\hspace{4cm}+\Delta^{2}(\Psi'(u))^{3})\biggr)\,du,\\
	&\lesssim\Delta\left(h_n^{-1}+\Delta h_n^{-2}+\Delta^{2}h_n^{-4}\right)\lesssim\Delta h_n^{-1}.
\end{split}
 \end{equation} 
To estimate the infimum $
\mathsf{E}\bigl[(\Delta X)_{1}^{2m}K_{m,n}^2\bigl(x-(\Delta X)_{1}\bigr)\bigr]$
 for $m=0,1,2,3$, we consider
\begin{equation*}
\begin{split}
\mathsf{E}\left[\left(\Delta X\right)_{1}^{2m}K_{m,n}^{2}\left(x-\left(\Delta X\right)_{1}\right)\right] & =\int_{\mathbb{R}}y^{2m}K_{m,n}^{2}\left(x-y\right)P_{\Delta}(dy)\\
 & =\int_{\mathbb{R}}K_{m,n}^{2}\left(x-y\right)g_{\Delta}^{2m}(y)dy\\
 & =\int_{\mathbb{R}}K_{m,n}^{2}\left(y\right)g_{\Delta}^{2m}(x-y)dy.
\end{split}
\end{equation*}
Since $K_{m,n}(z)=-\frac{1}{2\pi}\int_{\mathbb{R}}e^{-\mathrm{i}uz}Q_{m}(u)\varphi_{W}(uh_{n})\,du$, we have according to the Plancherel's theorem:
\begin{equation*}
\int_{\mathbb{R}}K_{m,n}^{2}(y)\,dy=-\frac{1}{2\pi}\int_{\mathbb{R}}\bigl|Q_{m}(u)\varphi_{W}(uh_{n})\bigr|^{2}du,
\end{equation*}
where
\begin{equation*}
\begin{split}
Q_{0}(u) & =\frac{\Delta}{\Phi_{\Delta X}(u)}\biggl(\frac{2-\alpha}{2}\bigl(\Delta(\Psi'(u))^{2}-\Psi''(u)\bigr)\biggr.\\
 & \biggl.\hspace{2cm}-u\frac{1-\alpha}{2}\bigl(\Psi'''(u)-3\Delta\Psi''(u)\Psi'(u)-\Delta^{2}(\Psi'(u))^{3}\bigr)\biggr),\\
Q_{1}(u) & =\frac{\Delta}{\Phi_{\Delta X}(u)}\biggl(3u\frac{1-\alpha}{2}\bigl(\Delta(\Psi'(u))^{2}+\Psi''(u)\bigr)-(2-\alpha)\Psi'(u)\biggr),\\
Q_{2}(u) & =\frac{1}{\Phi_{\Delta X}(u)}\biggl(\frac{2-\alpha}{2}-3u\Delta\Psi'(u)\frac{1-\alpha}{2}\biggr),\\
Q_{3}(u) & =\frac{1}{\Phi_{\Delta X}(u)}\biggl(u\frac{1-\alpha}{2}\biggr).
\end{split}
\end{equation*}
Furthermore applies
\begin{equation*}
\begin{split}
\inf_{x\in I}\int_{\mathbb{R}}K_{m,n}^{2}(y)g_{\Delta}^{2m}(x-y)\,dy & \geq\inf_{x\in I}g_{\Delta}^{2m}(x-y)\int_{\mathbb{R}}K_{m,n}^{2}(y)\,dy\\
 & \geq\frac{1}{2\pi}\inf_{x\in I}g_{\Delta}^{2m}(x-y)\int_{\mathbb{R}}\bigl|Q_{m}(u)\varphi_{W}(uh_{n})\bigr|^{2}du\\
 & =\frac{1}{2\pi}\inf_{x\in I}g_{\Delta}^{2m}(x-y)\int^{1/h_n}_{-1/h_n}\bigl|Q_{m}(u)\bigr|^{2}du.
\end{split}
\end{equation*}
Since $\bigl|\Phi_{\Delta X}(u)\bigr|\leq1$ for all \(u,\) it follows for $m=3:$
\begin{equation}
\label{BQ3}
\begin{split}
\int^{1/h_n}_{-1/h_n}\bigl|Q_{3}(u)\bigr|^{2}du & =\int^{1/h_n}_{-1/h_n}\frac{1}{\Phi_{\Delta X}^{2}(u)}\biggl(u\frac{1-\alpha}{2}\biggr)^{2}du \geq\int^{1/h_n}_{-1/h_n}\biggl(u\frac{1-\alpha}{2}\biggr)^{2}du\\
 & =\left(\frac{1-\alpha}{2}\right)^{2}\int^{1/h_n}_{-1/h_n}u^{2}\,du=\frac{(1-\alpha)^2}{2}\frac{u^{3}}{3}\mid_{0}^{1/h_n}=\mathcal{O}(h_n^{-3})
\end{split}
\end{equation}
The same argument applies to $m=2:$
\begin{equation*}
\begin{split}
\int^{1/h_n}_{-1/h_n}\left|Q_{2}(u)\right|^{2}du & =\int^{1/h_n}_{-1/h_n}\frac{1}{\Phi_{\Delta X}^{2}(u)}\left(\frac{2-\alpha}{2}-3u\Delta\Psi'(u)\frac{1-\alpha}{2}\right)^{2}du\\
 & \geq\left(\frac{2-\alpha}{2}\right)^{2}\int^{1/h_n}_{-1/h_n}du-\frac{3}{2}(2-\alpha)(1-\alpha)\Delta\int^{1/h_n}_{-1/h_n}u\Psi'(u)\,du\\
 & \hspace{2cm}+9\Delta^{2}\biggl(\frac{1-\alpha}{2}\biggr)^{2}\int^{1/h_n}_{-1/h_n}(u\Psi'(u))^{2}du.
\end{split}
\end{equation*}
Furthermore
\begin{equation*}
\begin{split}
\int^{1/h_n}_{-1/h_n}u\Psi'(u)du & =\frac{2}{1-\alpha}\biggl[\int^{1/h_n}_{-1/h_n}\biggl(\int_{0}^{1}u(-\sigma^{2}uy+\mathrm{i}\gamma)y^{\frac{\alpha}{1-\alpha}}\,dy\biggr)du\biggr.\\
 & \biggl.\hspace{1cm}+\int^{1/h_n}_{-1/h_n}\biggl(\int_{0}^{1}\int_{\mathbb{R}}\mathrm{i}ux(e^{\mathrm{i}uyx}-\mathbb{I}_{\{|x|\leq1\}})\nu(dx)y^{\frac{\alpha}{1-\alpha}}dy\biggr)du\biggr]\\
 & =\mathcal{O}(h_n^{-3}+h_n^{-2})=\mathcal{O}(h_n^{-3}).
\end{split}
\end{equation*}
Similarly  $\int^{1/h_n}_{-1/h_n}(u\Psi'(u))^{2}\,du=\mathcal{O}(h_n^{-5})$ and we have
\begin{equation}
\label{BQ2}
\int^{1/h_n}_{-1/h_n}\bigl|Q_{2}(u)\bigr|^{2}du =\mathcal{O}(h_n^{-1}+\Delta h_n^{-3}+\Delta^{2}h_n^{-5})=\mathcal{O}(h_n^{-1}),
\end{equation}
\begin{equation}
\label{BQ1}
\begin{split}
\int^{1/h_n}_{-1/h_n}\bigl|Q_{1}(u)\bigl|^{2}du & \geq\Delta^{2}\int^{1/h_n}_{-1/h_n}\biggl(3u\frac{1-\alpha}{2}\bigl(\Delta(\Psi'(u))^{2}+\Psi''(u)\bigr)-(2-\alpha)\Psi'(u)\biggr)^{2}du\\
 &=\mathcal{O}\bigl(\Delta^{2}(\Delta^2 h_n^{-7}+\Delta h_n^{-5}+h_n^{-3})\bigr)=\mathcal{O}(\Delta^{2}h_n^{-3}),
\end{split}
\end{equation}
\begin{equation}
\label{BQ0}
\begin{split}
\int^{1/h_n}_{-1/h_n}\bigl|Q_{0}(u)\bigr|^{2}du & =\Delta^{2}\int^{1/h_n}_{-1/h_n}\biggl(\frac{2-\alpha}{2}\bigl(\Delta(\Psi'(u))^{2}-\Psi''(u)\bigr)\biggr.\\
 & \biggl.\hspace{1em}-u\frac{1-\alpha}{2}\bigl(\Psi'''(u)-3\Delta\Psi''(u)\Psi'(u)-\Delta^{2}(\Psi'(u))^{3}\bigr)\biggr)^{2}du\\
 & =\mathcal{O}(\Delta^{2}(h_n^{-3}+\Delta h_n^{-3}+h_n^{-1}+\Delta^2 h_n^{-5}+
\Delta^3 h_n^{-7}+\Delta^4 h_n^{-9}))\\
&=\mathcal{O}(\Delta^{2}h_n^{-3}).
\end{split}
\end{equation}
Taking into account Lemma~\ref{InfPDelta}, we get
\begin{equation}
\label{InfPMoments}
\begin{split}
&\inf_{x\in I}\int_{\mathbb{R}}K_{0,n}^{2}(x-y)P_{\Delta}(dy)  \gtrsim\Delta^{2}h_n^{-3}\\
&\inf_{x\in I}\int_{\mathbb{R}}y^{2}K_{1,n}^{2}(x-y)P_{\Delta}(dy)  \gtrsim\Delta^{3}h_n^{-3}\\
&\inf_{x\in I}\int_{\mathbb{R}}y^{4}K_{2,n}^{2}(x-y)P_{\Delta}(dy)  \gtrsim\Delta h_n^{-1}\\
&\inf_{x\in I}\int_{\mathbb{R}}y^{6}K_{3,n}^{2}(x-y)P_{\Delta}(dy)  \gtrsim\Delta h_n^{-3}
\end{split}
\end{equation}
Finally, by combining \eqref{InfPMoments} with \eqref{SupK0}-\eqref{SupyK3},  we prove the claim.
\end{proof}
\subsection{Proof of Theorem \ref{thm1}}
Using the equations \eqref{ZerlDichte} and \eqref{rhoTilde}, the difference  $\widehat{\rho}_{n}(x) -\rho(x)$ can be represented as
\begin{equation*}
\widehat{\rho}_{n}(x)-\rho(x)=\underbrace{\bigl(\widehat{\rho}_{n}(x)-\widetilde{\rho}(x)\bigr)}_{R_n(x)}+\underbrace{\bigl(\widetilde{\rho}(x)-\rho(x)\bigr)}_{I_{\sigma^2_n}(x)+I_{\rho_n}(x)},
 \end{equation*}
 where
 \begin{equation*}
 \begin{split}
&\widetilde{\rho}(x):=-\frac{1}{2\pi\Delta}\int_{\mathbb{R}}e^{-\mathrm{i}ux}\bigl[(\mathcal{L}^{-1}_\alpha \Psi)''(u)+\Delta\sigma^{2}\bigr]\varphi_{W}(uh)\,du,\\
&I_{\sigma^2_n}(x):=-\frac{1}{2\pi}\int_{\mathbb{R}}e^{-\mathrm{i}u x}(\widehat{\sigma}^{2}_{n}-\sigma^{2})\varphi_{W}(uh)\,du,\\
&I_{\rho_n}(x):=\bigl[\rho\ast(h^{-1}W(\cdot/h))\bigr](x)-\rho(x).
\end{split}
 \end{equation*}
Under  assumptions \ref{assumption1} \ref{ass6} and lemma \ref{2TermRest}, the   terms $I_{\rho_n}$ and $I_{\sigma^2_n}$ are asymptotically (as \(n\to \infty\) and \(\Delta\to 0\)) smaller than $R_n$ and hence can be neglected  when constructing the confidence interval for the transformed L\'evy density $\rho.$ With the  notations \ref{Rn} for $R_n(x)$, namely
\begin{equation*}
\begin{split}
R_{n}(x)&=\frac{1}{n\Delta}\stackrel[m=0]{3}{\sum}\biggl(\stackrel[j=1]{n}{\sum}\bigl\{ \mathrm{i}^{m}(\Delta X)_{j}^{m}K_{m,n}(x-(\Delta X)_{j})\bigr.\biggr.\\
&\biggl.\bigl.\hspace{3cm}-\mathsf{E}\bigl[\mathrm{i}^{m}(\Delta X)_{1}^{m}K_{m,n}(x-(\Delta X)_{1})\bigr]\bigr\}\biggr),
\end{split}
\end{equation*}
where the kernel functions $K_{m,n}(z),\,m=0,1,2,3,$ are defined as
\begin{equation*}
K_{m,n}(z):=-\frac{1}{2\pi}\int_{\mathbb{R}}e^{-\mathrm{i}uz}Q_{m}(u)\varphi_{W}(uh_{n})\,du
\end{equation*}
 consider the process
\begin{equation}
T_{n}(x):= \frac{\sqrt{n}\Delta}{s(x)}R_{n}(x),
\label{Tn}
\end{equation}
where $s^{2}(x)$ is given by
\begin{equation}
\label{VarRn}
s^{2}(x):= \mathsf{Var}\biggl[\stackrel[m=0]{3}{\sum}\mathrm{i}^{m}(\Delta X)_{1}^{m}K_{m,n}(x-(\Delta X)_{1})\biggr]
\end{equation}
Futher we show that exists a tight $\ell^{\infty}(I)$-sequence 
of Gaussian random variables $T_{n}^{G}$  with zero mean  and the same covariance function as one of $\,T_{n}$, and such that the distribution of $\Vert  T_{n}^{G}\Vert_{I}:=\sup_{x\in I}|T_{n}^{G}(x)|$
asymptotically approximates the distribution of $\Vert T_{n}\Vert_{I}$
in the sense that
\begin{equation*}
\underset{z\in\mathbb{R}}{\sup}\bigl|\mathrm{P}\left\{ \bigl\Vert T_{n}\bigr\Vert _{I}\leq z\right\} -\mathrm{P}\left\{ \bigl\Vert T_{n}^{G}\bigr\Vert_{I}\leq z\right\} \bigr|\to 0, \quad n\to \infty.
\end{equation*}
In what follows, we always assume Assumption \ref{assumption1}. The proofs rely on modern empirical process theory. For a probability measure $Q$ on a measurable space $(S, \mathcal{S})$ and a class of measurable functions $\mathcal{F}$ on $S$ such that $\mathcal{F} \in  L^2(Q)$, let $N(\mathcal{F}, \left\Vert \cdot \right\Vert_{Q,2},\varepsilon)$ denote the $\varepsilon$-covering number for $\mathcal{F}$
with respect to the $L^2(Q)$-seminorm $\left\Vert \cdot \right\Vert_{Q,2}$. See Section 2.1 in \cite{van1996weak} for details. Let $\stackrel{d}{=}$ denote the equality in distribution.
Consider the function class 
\begin{equation*}
\mathcal{F}_{i,n}=\left\{y\mapsto \frac{(\mathrm{i} y)^i}{s(x)}K_{i,n}(x-y):\quad x\in I\right\} .
\end{equation*}
According to Lemma~\ref{lemvar} 
\begin{equation*}
\inf_{x\in I}s^{2}(x)\gtrsim\Delta h_n^{-3}
\end{equation*}
and we have
\begin{equation}
\label{TnKonv}
\frac{\sqrt{n}\Delta(\widehat{\rho}_{n}(x)-\rho(x))}{s(x)}=T_{n}(x)+o_{\mathrm{P}}\left(h_n^{1/2}\log h_n^{-1}\right) 
\end{equation}
uniformly in $x\in I$. Under condition  \ref{ass5} of Assumption~\ref{assumption1}, the expression $h_n^{1/2}\log h_n^{-1}$  converges to 0.
Further we approximate $\Vert T_{n}\Vert_{I}$ by the supremum of a tight Gaussian random variable $T_{n}^{G}$ in $\ell^{\infty}(I)$ with  expected value zero and the same covariance function as for the random variable $T_{n}$. 
Using Theorem~2.1 in \cite{chernozhukovand2016}, which proves the existence of such  random variable $T_{n}^{G}$, we consider the empirical process:
\begin{equation*}
\mathsf{G}_{n}(f_{i,n})=\frac{1}{\sqrt{n}}\stackrel[i=0]{3}{\sum}\left[\stackrel[j=1]{n}{\sum}\bigl\{f_{i,n}((\Delta X)_{j})-\mathsf{E}\bigl[f_{i,n}((\Delta X)_{1})\bigr]\bigr\}\right],\ f_{i,n}\in\mathcal{F}_{i,n}. 
\end{equation*}
Note that according lemma \ref{InfPDelta} the increment process $ X_{\Delta}-X_{0}$ has the distribution $P_{\Delta}$, so that $y^{2m}P_{\Delta}(dy)=g_{\Delta}^{2m}(y) dy,$  with $\bigl\Vert g_{\Delta}^{2}\bigr\Vert_{\mathbb{R}}\lesssim\Delta^{1/2}$ and $\bigl\Vert g_{\Delta}^{2m}\bigr\Vert_{\mathbb{R}}\lesssim\Delta$ for $m=2,3.$ Hence
\begin{equation*}
\begin{split}
\mathsf{E}\bigl[(\Delta X)_{1}^{2m}K_{m,n}^{2}(x-(\Delta X)_{1})\bigr] & =\int_{\mathbb{R}}y^{2m}K_{m,n}^{2}(x-y)P_{\Delta}(dy)\\
 & =\int_{\mathbb{R}}K_{m,n}^{2}(x-y)g_{\Delta}^{2m}(y)\,dy\\
 & =\int_{\mathbb{R}}K_{m,n}^{2}(y)g_{\Delta}^{2m}(x-y)\,dy\\
 &\leq \bigl\Vert g_{\Delta}^{2m}\bigr\Vert_{\mathbb{R}}\int_{\mathbb{R}}K_{m,n}^{2}(y)\,dy
\end{split}
\end{equation*}
Since $K_{m,n}(z)=-\frac{1}{2\pi}\int_{\mathbb{R}}e^{-\mathrm{i}uz}Q_{m}(u)\varphi_{W}(uh_{n})\,du$, we have according to the Plancherel's theorem,
\begin{equation*}
\int_{\mathbb{R}}K_{m,n}^{2}(y)\,dy=-\frac{1}{2\pi}\int_{\mathbb{R}}\bigl|Q_{m}(u)\varphi_{W}(uh_{n})\bigr|^{2}du.
\end{equation*}
Using   \eqref{BQ3}, \eqref{BQ2}, \eqref{BQ1} and \eqref{BQ0}, we get that
\begin{equation}
\begin{split}
&\mathsf{E}\bigl[K_{0,n}^{2}(x-(\Delta X)_{1})\bigr] \lesssim \Delta^2 h_n^{-3},\\
&\mathsf{E}\bigl[(\Delta X)_{1}^{2}K_{1,n}^{2}(x-(\Delta X)_{1})\bigr]\lesssim \Delta^{5/2} h_n^{-3}, \\
&\mathsf{E}\bigl[(\Delta X)_{1}^{4}K_{2,n}^{2}(x-(\Delta X)_{1})\bigr]\lesssim \Delta h_n^{-1}, \\
&\mathsf{E}\bigl[(\Delta X)_{1}^{6}K_{3,n}^{2}(x-(\Delta X)_{1})\bigr] \lesssim \Delta h_n^{-3},
\end{split}
\end{equation}
 holds and it also follows that 
 \begin{equation*}
 \underset{f_{i,n}\in\mathcal{F}_{i,n}}{\sup}\mathsf{E}\biggl[\stackrel[i=1]{3}{\sum}f_{i,n}^2((\Delta X)_{1})\biggr]\lesssim \frac{\Delta h_n^{-3}}{\Delta h_n^{-3}}\lesssim1.
 \end{equation*}
Furthermore 
 \begin{equation}
 \label{Ksupp}
 \begin{split}
  \bigl\Vert K_{0,n}(\cdot-(\Delta X)_{1})\bigr\Vert_{\mathbb{R}}&\lesssim \Delta (\Delta h_n^{-2}+1+h_n^{-1}+\Delta^2 h_n^{-4})\lesssim \Delta h_n^{-1},\\
 \bigl\Vert K_{1,n}(\cdot-(\Delta X)_{1})\bigr\Vert_{\mathbb{R}}&\lesssim \Delta (\Delta h_n^{-3}+h_n^{-1})\lesssim \Delta h_n^{-1},\\
 \bigl\Vert K_{2,n}(\cdot-(\Delta X)_{1})\bigr\Vert_{\mathbb{R}}&\lesssim  1+\Delta h_n^{-2}\lesssim  1,\\
 \bigl\Vert K_{3,n}(\cdot-(\Delta X)_{1})\bigr\Vert_{\mathbb{R}}&\lesssim  h_n^{-1}.
 \end{split}
 \end{equation}
Hence  $\underset{f_{i,n}\in\mathcal{F}_{i,n}}{\sup}\biggl\Vert \stackrel[i=0]{3}{\sum}f_{i,n}((\Delta X)_{1})\biggr\Vert _{\mathbb{R}}\lesssim \frac{h_n^{-1}}{\sqrt{\Delta h_n^{-3}}}\lesssim \sqrt{h_n}/\sqrt{\Delta}$  and  we have 
\begin{equation}
\label{supEf}
\begin{split}
&\underset{f_{i,n}\in\mathcal{F}_{i,n}}{\sup}\mathbb{E}\biggl[\stackrel[i=0]{3}{\sum}f_{i,n}^3((\Delta X)_{1})\biggr]\lesssim\underset{f_{i,n}\in\mathcal{F}_{i,n}}{\sup}\mathbb{E}\biggl[\stackrel[i=0]{3}{\sum}f_{i,n}^2((\Delta X)_{1})\biggr]\sqrt{h_n}/\sqrt{\Delta}\lesssim\sqrt{h_n}/\sqrt{\Delta},\\
&\underset{f_{i,n}\in\mathcal{F}_{i,n}}{\sup}\mathbb{E}\biggl[\stackrel[i=0]{3}{\sum}f_{i,n}^4((\Delta X)_{1})\biggr]\lesssim\underset{f_{i,n}\in\mathcal{F}_{i,n}}{\sup}\mathbb{E}\biggl[\stackrel[i=0]{3}{\sum}f_{i,n}^2((\Delta X)_{1})\biggr]h_n/\Delta \lesssim h_n/\Delta.\\
\end{split}
\end{equation}
Due Theorem 2.1 in \cite{chernozhukovand2016} with $B(f)\equiv0, A\lesssim1, v\lesssim1, \sigma\sim1, b\lesssim\frac{\sqrt{h_n}}{\sqrt{\Delta}}, \gamma\lesssim\frac{1}{\log n}$ and $q$ sufficiently large, we derive that there exist a random variable $V_{n}$ with the same distribution as $\bigl\Vert U_{n}\bigr\Vert_{\mathcal{F}_{i,n}}$  such that
\begin{equation}
\bigl|\bigl\Vert\mathsf{G}_{n}\bigr\Vert_{\mathcal{F}_{i,n}}-V_{n}\bigr|=\mathcal{O}_{P}\biggl\{ \frac{(\log n)^{1+1/q}}{n^{1/2-1/q}\sqrt{\Delta h^{-1}_n}}+\frac{\log n}{(n\Delta h^{-1}_n)^{1/6}}\biggr\} =\mathcal{O}_{P}\biggl\{\frac{\log n}{(n\Delta h^{-1}_n)^{1/6}}\biggr\}.
\label{GT}
\end{equation}
Taking into account Assumption~ \ref{assumption1} \ref{ass5}, the expression \eqref{TnKonv} converges to zero  slower than \eqref{GT}. Then the statement \ref{DifT} follows. In addition, for
\begin{equation*}
f_{j,n}(y)=(\mathrm{i}y)^j K_{j,n}(x-y)/s(x)
\end{equation*}
 we define $T_{n}^{G}(x)=U_{n}(\stackrel[i=0]{3}{\sum}f_{i,n}(x)),\ x\in I$, and we observe, that there exists a tight Gaussian random variable $T_{n}^{G}(x)$   in  $\ell^{\infty}(I)$  with  expected value zero and the same covariance function as for $T_{n}$. The following concentration inequality holds (see Theorem 2.1 in \cite{chernozhukov2014anti} for any $\varepsilon>0,$
\begin{equation}
\underset{z\in\mathbb{R}}{\sup}\,\mathrm{P}\bigl\{ \bigl|\bigl\Vert T_{n}^{G}\bigr\Vert_{I}-z\bigr|\leq\varepsilon\bigr\} \leq4\varepsilon\bigl(1+\mathsf{E}\bigl[\bigl\Vert T_{n}^{G}\bigr\Vert_{I}\bigr]\bigr).
\label{anticon}
\end{equation}
According to the Corollary 2.1 in \cite{chernozhukov2014anti}, Theorem~3 in \cite{chernozhukov2015comparison} and the representation \ref{GT} for $\Vert\mathsf{G}_{n}\Vert_{\mathcal{F}_{i,n}}-V_{n},$  we can claim that there exists a sequence $\varepsilon_{n}\rightarrow0$  such 
\begin{equation*}
\mathrm{P}\biggl\{ \bigl|\bigl\Vert\mathsf{G}_{n}\bigr\Vert_{\mathcal{F}_{n}}-V_{n}\bigr|\geq\varepsilon_{n}\left(h_n^{1/2}\log h_n^{-1}\right)\biggr\}\leq\varepsilon_{n}.
\end{equation*}
Since $\bigl\Vert\mathsf{G}_{n}\bigr\Vert_{\mathcal{F}_{i,n}}=\bigl\Vert T_{n}\bigr\Vert_{I}$ and $V_{n}\overset{d}{=}\bigl\Vert T_{n}^{G}\bigr\Vert_{I},$ we have that
\begin{equation*}
\begin{split}\mathrm{P}\bigl\{\bigl\Vert T_{n}\bigr\Vert_{I}\leq z\bigr\}  & \leq\mathrm{P}\biggl\{ \bigl\Vert T_{n}^{G}\bigr\Vert_{I}\leq z+\varepsilon_{n}\left(h_n^{1/2}\log h_n^{-1}\right)\biggr\} +\varepsilon_{n}\leq\\
 & \leq\mathrm{P}\bigl\{ \bigl\Vert T_{n}^{G}\bigr\Vert_{I}\leq z\bigr\} +4\varepsilon_{n}\left(h_n^{1/2}\log h_n^{-1}\right)\bigl(1+\mathsf{E}\bigl[\bigl\Vert T_{n}^{G}\bigr\Vert_{I}\bigr]\bigr)+\varepsilon_{n},
\end{split}
\end{equation*}
for all $z\in\mathbb{R}$. In this way we have
\begin{equation*}
\mathrm{P}\bigl\{ \bigl\Vert T_{n}\bigr\Vert_{I}\leq z\bigr\} \geq\mathrm{P}\bigl\{ \bigl\Vert T_{n}^{G}\bigr\Vert_{I}\leq z\bigr\} -4\varepsilon_{n}\left(h_n^{1/2}\log h_n^{-1}\right)\bigl(1+\mathsf{E}\bigl[\bigl\Vert T_{n}^{G}\bigr\Vert_{I}\bigr]\bigr)-\varepsilon_{n},
\end{equation*}
for all $z\in\mathbb{R}$. 

It follows from Corollary 2.2.8 (see \cite{van1996weak}) together with  $\mathrm{Var}\bigl[\stackrel[i=0]{3}{\sum}f_{i,n}((\Delta X)_{1})\bigr]=1$, that
\begin{equation*}
\mathsf{E}\bigl[\bigl\Vert T_{n}^{G}\bigr\Vert_{I}\bigr]=\mathsf{E}\bigl[\bigl\Vert U_{n}\bigr\Vert_{\mathcal{F}_{i,n}}\bigr]\lesssim\int^1_0\sqrt{1+\log\bigl(1/\varepsilon\sqrt{\Delta h_n^{-1}}\bigr)}\,d\varepsilon\lesssim(\log n)^{1/2}.
\end{equation*}
The proof is completed.
\subsection{Proof of Theorem~\ref{theorem2}}
The proof scheme of the validity of bootstrap confidence bands was introduced by Kato and Kurisu  \cite{kato2017bootstrap} and can be represented as follows.

Step 1: Conditional distribution of the supremum of the multiplier process  $\bigl\Vert \widehat{T_{n}}^{MB}\bigr\Vert_{I}$  consistently estimates the distribution of the Gaussian supremum  $\bigl\Vert T_{n}^{G}\bigr\Vert_{I}$  in the sense that
\begin{equation*}
\underset{z\in\mathbb{R}}{\sup}\biggl|\mathrm{P}\biggl\{\bigl\Vert \widehat{T_{n}}^{MB}\bigr\Vert_{I}\leq z\mid\mathsf{D}_{n}\biggr\}-\mathrm{P}\bigl\{\bigl\Vert T_{n}^{G}\bigr\Vert_{I}\leq z\bigr\}\biggr|=o_{P}(1)
\end{equation*}
Step 2: In addition together with Theorem \ref{thm1} we have that
\begin{equation*}
\underset{z\in\mathbb{R}}{\sup}\biggl|\mathrm{P}\biggl\{\biggl\Vert \frac{\Delta\sqrt{n}(\widehat{\rho_{n}}(\cdot)-\rho(\cdot))}{\widehat{s}_{n}(\cdot)}\biggr\Vert _{I}\leq z\biggr\} -\mathrm{P}\bigl\{ \bigl\Vert T_{n}^{G}\bigr\Vert_{I}\leq z\bigr\} \biggr|\to0.
\end{equation*}
Step 3: Combining steps 1 and 2 leads to the conclusion of Theorem \ref{theorem2}.
For an precise proof of Theorem \ref{theorem2} we need the following technical lemma.
\begin{lem} 
We have
 \begin{equation*}
 \bigl\Vert\widehat{s}_{n}^{2}(\cdot)/s^{2}(\cdot)-1\bigr\Vert _{I}=o_{P}\bigl\{\bigl((n\Delta h_n^{-1})^{-1}\log n\bigr)^{1/2}\bigr\}.
\end{equation*}  
\label{sdachs}
\end{lem}
This Lemma can be proved using the technique in \cite{kato2017bootstrap} (see Lemma 8.10) together with Corollary  5.1 and A.1 in \cite{chernozhukovand2014a}.
\begin{proof}
First using Lemma \ref{dPhiest} note that
\begin{align*}
\biggl\Vert \frac{\mathbb{D}}{\widehat{\Phi}_{\Delta X}\Phi_{\Delta X}}\biggr\Vert _{\mathbb{R}} \quad\quad& \lesssim n^{-1/2}\log h_n^{-1},\\
\biggl\Vert \biggl(\frac{\mathbb{D}}{\widehat{\Phi}_{\Delta X}\Phi_{\Delta X}}\biggr)'\biggr\Vert _{\mathbb{R}}\,\, & \leq\biggl\Vert \frac{\mathbb{D}'}{\widehat{\Phi}_{\Delta X}\Phi_{\Delta X}}+\frac{\Delta\mathbb{D}\Psi'}{\widehat{\Phi}_{\Delta X}\Phi_{\Delta X}}\biggr\Vert _{\mathbb{R}}\\
 & \lesssim n^{-1/2}\log h_n^{-1}(\Delta^{1/2}+\Delta h_n^{-1})\lesssim n^{-1/2}\log h_n^{-1}\Delta^{1/2},\\
 \biggl\Vert \biggl(\frac{\mathbb{D}}{\widehat{\Phi}_{\Delta X}\Phi_{\Delta X}}\biggr)''\biggr\Vert _{\mathbb{R}}\, & \lesssim n^{-1/2}\log h_n^{-1}(\Delta^{1/2}+\Delta^{3/2}h_n^{-1}+\Delta^{2}h_n^{-2})\\
 & \lesssim n^{-1/2}\log h_n^{-1}\Delta^{1/2},\\
\biggl\Vert \biggl(\frac{\mathbb{D}}{\widehat{\Phi}_{\Delta X}\Phi_{\Delta X}}\biggr)'''\biggr\Vert _{\mathbb{R}} & \lesssim n^{-1/2}\log h_n^{-1}(\Delta^{1/2}+\Delta^{3/2}h_n^{-1}+\Delta^{2}h_n^{-2}+\Delta^{3}h_n^{-3})\\
 & \lesssim n^{-1/2}\log h_n^{-1}\Delta^{1/2}.
\end{align*} 
Furthermore
\begin{align*}
\widehat{K}_{3,n}(z)-K_{3,n}(z) & =-\frac{1}{2\pi}\int_{\mathbb{R}}e^{-\mathrm{i}uz}\bigl(\widehat{Q}_{3}(u)-Q_{3}(u)\bigr)\varphi_{W}(uh_{n})\,du\\
 & =-\frac{1-\alpha}{4\pi}\int_{\mathbb{R}}e^{-\mathrm{i}uz}u\biggl(\frac{-\mathbb{D}(u)}{\widehat{\Phi}_{\Delta X}(u)\Phi_{\Delta X}(u)}\biggr)\varphi_{W}(uh_{n})\,du
\end{align*}
and 
\begin{eqnarray*}
\bigl\Vert \widehat{K}_{3,n}-K_{3,n}\bigr\Vert _{\mathbb{R}}  \lesssim n^{-1/2}\log h_n^{-1}\int^{1/h_n}_{-1/h_n}e^{-\mathrm{i} ux}u\,du\lesssim n^{-1/2}h_n^{-1}\log h_n^{-1}.
\end{eqnarray*}
Analogously
\begin{equation*}
\begin{split}
\bigl\Vert \widehat{K}_{2,n}-K_{2,n}\bigr\Vert _{\mathbb{R}}\quad&\lesssim n^{-1/2}\log h_n^{-1}\biggl(h_n^{-1}+\Delta^{1/2}\int^{1/h_n}_{-1/h_n}e^{-\mathrm{i} ux}u\,du\biggr)\\
&\lesssim n^{-1/2}h_n^{-1}\log h_n^{-1},
\\
\bigl\Vert \widehat{K}_{1,n}-K_{1,n}\bigr\Vert _{\mathbb{R}}\quad&\lesssim n^{-1/2}\log h_n^{-1}\Delta^{1/2}(h_n^{-1}+\Delta^{1/2}h_n^{-1})\lesssim n^{-1/2}\Delta^{1/2}h_n^{-1}\log h_n^{-1},\\
\bigl\Vert \widehat{K}_{0,n}-K_{0,n}\bigr\Vert_{\mathbb{R}}\quad&\lesssim n^{-1/2}\log h_n^{-1}\Delta^{1/2}(h_n^{-1}+\Delta^{1/2}h_n^{-1})\lesssim n^{-1/2}\Delta^{1/2}h_n^{-1}\log h_n^{-1}.
\end{split}
\end{equation*}
Since $\bigl\Vert \widehat{K}_{i,n}^2-K_{i,n}^2\bigr\Vert_{\mathbb{R}}\leq\bigl\Vert \widehat{K}_{i,n}-K_{i,n}\bigr\Vert_{\mathbb{R}}\bigl\Vert \widehat{K}_{i,n}+K_{i,n}\bigr\Vert_{\mathbb{R}},$ we have according to \eqref{Ksupp},
\begin{equation*}
\begin{split}
\bigl\Vert \widehat{K}_{0,n}^2-K_{0,n}^2\bigr\Vert_{\mathbb{R}}&\lesssim n^{-1/2}\Delta^{3/2}h_n^{-2}\log h_n^{-1},\\
\bigl\Vert \widehat{K}_{1,n}^2-K_{1,n}^2\bigr\Vert_{\mathbb{R}}&\lesssim n^{-1/2}\Delta^{3/2}h_n^{-2}\log h_n^{-1},\\
\bigl\Vert \widehat{K}_{2,n}^2-K_{2,n}^2\bigr\Vert_{\mathbb{R}}&\lesssim n^{-1/2}h_n^{-1}\log h_n^{-1},\\
\bigl\Vert \widehat{K}_{3,n}^2-K_{3,n}^2\bigr\Vert_{\mathbb{R}}&\lesssim n^{-1/2}h_n^{-2}\log h_n^{-1}.
\end{split}
\end{equation*}
Next we have for $i=0,1,2,3$
\begin{equation}
\begin{split}
    &\biggl\Vert\frac{1}{n}\stackrel[j=1]{n}{\sum}(\Delta X)_j^i\bigl\{\widehat{K}_{i,n}\bigl(\cdot-(\Delta X)_j\bigr)-K_{i,n}\bigl(\cdot-(\Delta X)_j\bigr)\bigr\}\biggr\Vert_{I} \\
    &\hspace{4.5cm}=\mathcal{O}_{p}\bigl(n^{-1/2}\Delta^{(i+1/2)\wedge1} h_n^{-1}\log h_n^{-1}\bigr).
 \end{split}
\label{K1dif}
\end{equation} 
and analogously
\begin{equation}
\begin{split}
 &\biggl\Vert\frac{1}{n}\stackrel[j=1]{n}{\sum}\bigl\{\widehat{K}^2_{0,n}\bigl(\cdot-(\Delta X)_j\bigr)-K^2_{0,n}\bigl(\cdot-(\Delta X)_j\bigr)\bigr\}\biggr\Vert_{I} =O_{p}\bigl(n^{-1/2}\Delta^{3/2}h_n^{-2}\log h_n^{-1}\bigr),\\
 & \biggl\Vert\frac{1}{n}\stackrel[j=1]{n}{\sum}(\Delta X)^2_j\bigl\{\widehat{K}^2_{1,n}\bigl(\cdot-(\Delta X)_j\bigr)-K^2_{1,n}\bigl(\cdot-(\Delta X)_j\bigr)\bigr\}\biggr\Vert_{I} =O_{p}\bigl(n^{-1/2}\Delta^{5/2} h_n^{-2}\log h_n^{-1}\bigr),\\
   &\biggl\Vert\frac{1}{n}\stackrel[j=1]{n}{\sum}(\Delta X)^4_j\bigl\{\widehat{K}^2_{2,n}\bigl(\cdot-(\Delta X)_j\bigr)-K^2_{2,n}\bigl(\cdot-(\Delta X)_j\bigr)\biggr\}\biggr\Vert_{I} =O_{p}\bigl(n^{-1/2}\Delta h_n^{-1}\log h_n^{-1}\bigr),\\
   & \biggl\Vert\frac{1}{n}\stackrel[j=1]{n}{\sum}(\Delta X)_j^6\bigl\{\widehat{K}^2_{3,n}\bigl(\cdot-(\Delta X)_j\bigr)-K^2_{3,n}\bigl(\cdot-(\Delta X)_j\bigr)\bigr\}\biggr\Vert_{I} =O_{p}\bigl(n^{-1/2}\Delta h_n^{-2}\log h_n^{-1}\bigr).
 \end{split}
\label{K2dif}
\end{equation} 
By the previous statement, we conclude that 
\begin{equation*}
\widehat{s}^{2}_{n}(x)=\widetilde{s}^{2}_{n}(x)+\mathcal{O}_{p}\bigl(n^{-1/2}\Delta h_n^{-2} \log h_n^{-1}\bigr)
\end{equation*}
uniformly in $x\in I $,  where
\begin{equation*}
\begin{split}
\widetilde{s}^{2}_{n}(x)&:=\stackrel[m=0]{3}{\sum}\biggl(\frac{1}{n}\stackrel[j=0]{n}{\sum}\bigl[(\Delta X)_{j}^{2m}K_{m,n}^{2}\bigl(x-(\Delta X)_{j}\bigr)\bigr]\biggr. \\
&\biggl. \hspace{2cm} -\left\{ \frac{1}{n}\stackrel[j=0]{n}{\sum}\bigl[(\Delta X)_{j}^{m}K_{m,n}\bigl(x-(\Delta X)_{j}\bigr)\bigr]\right\} ^{2}\biggr).
\end{split}
\end{equation*}
Moreover, since $\inf_{x\in I}s^2(x)\apprge\Delta h_n^{-3}$ we obtain that 
\begin{equation*}
\frac{n^{-1/2}\Delta h_n^{-2}\log h_n^{-1}}{\Delta h_n^{-3}}=n^{-1/2}h_n\log h_n^{-1}\ll \sqrt{\frac{\log n}{n\Delta h_n^{-1}}}.
\end{equation*}
Finally let us prove that $\bigl\Vert\widetilde{s}_{n}^{2}(\cdot)/s^{2}(\cdot)-1\bigr\Vert _{I}=o_{P}\bigl\{\bigl((n\Delta h_n^{-1})^{-1}\log n\bigr)^{1/2}\bigr\}.$ It follows from \eqref{SupK0}, \eqref{SupyK1}, \eqref{SupyK2} und \eqref{SupyK3},
\begin{equation*}
\begin{split}
&\bigl\Vert \mathsf{E}\bigl[K_{0,n}\bigl(\cdot-(\Delta X)_{1}\bigl)/s(\cdot)\bigr]\bigr\Vert_{I}\lesssim\Delta h_n^{-1}/\sqrt{\Delta h_n^{-3}}\lesssim (\Delta h_n)^{1/2}\\
&\bigl\Vert \mathsf{E}\bigl[(\Delta X)_{1}K_{1,n}\bigl(\cdot-(\Delta X)_{1}\bigr)/s(\cdot)\bigr]\bigr\Vert_{I}\lesssim\Delta/\sqrt{\Delta h_n^{-3}}\lesssim\Delta^{1/2}h_n^{3/2} \\
&\bigl\Vert \mathsf{E}\bigl[(\Delta X)_{1}^{2}K_{2,n}\bigl(\cdot-(\Delta X)_{1}\bigr)/s(\cdot)\bigr]\bigr\Vert_{I}\lesssim\Delta/\sqrt{\Delta h_n^{-3}}\lesssim \Delta^{1/2}h_n^{3/2}\\
&\bigl\Vert \mathsf{E}\bigl[(\Delta X)_{1}^{3}K_{3,n}\bigl(\cdot-(\Delta X)_{1}\bigr)/s(\cdot)\bigr]\bigr\Vert_{I}\lesssim\Delta h_n^{-1}/\sqrt{\Delta h_n^{-3}}\lesssim (\Delta h_n)^{1/2}.
\end{split}
\end{equation*}
Note that
\begin{equation*}
\sup_{f_{i,n}\in\mathcal{F}_{i,n}^{2}}\mathsf{E}\biggl[\stackrel[i=0]{3}{\sum}f^{2}_{i,n}((\Delta X)_{1})\biggr]\lesssim\sup_{f_{i,n}\in\mathcal{F}_{i,n}}\mathsf{E}\biggl[\stackrel[i=0]{3}{\sum}f^{4}_{i,n}((\Delta X)_{1})\biggr]\lesssim h_n/\Delta
\end{equation*} 
and
\begin{equation*}
\sup_{f_{i,n}\in\mathcal{F}_{i,n}^{2}}\left\Vert \stackrel[i=0]{3}{\sum}f_{i,n} \right\Vert _{\mathbb{R}}\lesssim\sup_{f_{i,n}\in\mathcal{F}_{i,n}}\left\Vert \stackrel[i=0]{3}{\sum}f_{i,n}^{2}\right\Vert _{\mathbb{R}}\lesssim h_n/\Delta.
\end{equation*}
Together with Corollary 5.1 in \cite{chernozhukovand2014a} and Theorem 2.14.1 in \cite{van1996weak} we get   
 \begin{equation*}
\begin{split}
&\biggl\Vert\frac{1}{n}\stackrel[i=0]{3}{\sum}\biggl[\stackrel[j=1]{n}{\sum}\bigl(f_{i,n}^{2}((\Delta X)_{j})-\mathsf{E}\bigl[f_{i,n}^{2}((\Delta X)_{j})\bigr]\bigr)\biggr]\biggr\Vert _{\mathcal{F}_{i,n}}\lesssim \sqrt{\frac{\log n}{n\Delta h_n^{-1}}}+\frac{\log n}{n\Delta h_n^{-1}}\lesssim \sqrt{\frac{\log n}{n\Delta h_n^{-1}}},\\
 & \biggl\Vert \frac{1}{n}\stackrel[i=0]{3}{\sum}\biggl[\stackrel[j=1]{n}{\sum}\bigl(f_{i,n}((\Delta Z)_{j})-\mathsf{E}\bigl[f_{i,n}((\Delta Z)_{j})\bigr])\biggr]\biggr\Vert _{\mathcal{F}_{i,n}}\lesssim \sqrt{\frac{1}{n\Delta h_n^{-1}}}\ll\sqrt{\frac{\log n}{n\Delta h_n^{-1}}}.
\end{split}
\end{equation*}
Finally we have $\bigl\Vert\widetilde{s}_{n}^{2}(\cdot)/s^{2}(\cdot)-1\bigr\Vert _{I}=\mathcal{O}_{P}\bigl((n\Delta h_n^{-1})^{-1}\log n\bigr)^{1/2}.$
This completes the proof.
\end{proof}
We get from (\ref{K1dif})-(\ref{K2dif}), 
 \begin{equation*}
 \begin{split}
& \biggl\Vert \biggl(\stackrel[j=1]{n}{\sum}\omega_{j}\biggr)\biggl\{\frac{1}{n}\stackrel[j=1]{n}{\sum}(\Delta X)_j^m\bigl\{\widehat{K}_{m,n}\bigl(\cdot-(\Delta X)_j\bigr)-K_{m,n}\bigl(\cdot-(\Delta X)_j\bigr)\bigr\}\biggr\}\biggr\Vert_{I}\\
 &\hspace{7cm}=\mathcal{O}_{p}\bigl(\Delta^{1/2} h_n^{-1}\log h_n^{-1}\bigr).
 \end{split}
\end{equation*}
Hence
\begin{equation*}
\begin{split}
&\biggl\Vert\stackrel[j=1]{n}{\sum}\omega_{j}\bigl\{ (\Delta X)_j^m\bigl\{\widehat{K}_{m,n}\bigl(\cdot-(\Delta X)_j\bigr)-K_{m,n}\bigl(\cdot-(\Delta X)_j\bigr)\bigr\}\bigr\}\biggr\Vert_{I}\\
&\hspace{1cm}\leq \mathcal{O}_{p}\bigl(n^{-1/2}h_n^{-1}\log h_n^{-1}\bigr)\mathsf{E}\biggl[\sqrt{\stackrel[j=1]{n}{\sum}(\Delta X)_{j}^{2m}}\,\,\biggr]=\mathcal{O}_{p}\bigl(\Delta^{1/2}h_n^{-1}\log h_n^{-1}\bigr).
\end{split}
\end{equation*}
Since $\bigl\Vert 1/\widehat{s}_{n}(x)\bigr\Vert _{I}=\mathcal{O}_{p}\bigl(1/\sqrt{\Delta h_n^{-3}}\bigr),$ using Lemma \ref{sdachs}  we obtain 
\begin{equation}
\widehat{T}_{n}^{MB}(x)=\biggl[1+o_{p}\biggl(\sqrt{\frac{\log n}{n\Delta h_n}}\biggr)\biggr]\biggl(T_{n}^{MB}(x)+\mathcal{O}_{p}\left( h_n^{1/2}\log h_n^{-1 }\right)\biggr) 
\label{TMB}
\end{equation}
Applying Theorem 2.2 in \cite{chernozhukovand2016} with $B(f)\equiv0, A\lesssim1, v\lesssim1, \sigma\sim1, b\lesssim\frac{\sqrt{h_n}}{\sqrt{\Delta}}, \gamma\lesssim\frac{1}{\log n}$ and sufficiently large $q,$ we conclude that there exists a random variable $V_{n}^{\xi}$ whose conditional distribution given  $\mathsf{D}_{n}$ is identical to the distribution of $\bigl\Vert U_{n}\bigr\Vert_{\mathcal{F}_{i,n}}$, that is, $\mathrm{P}\bigl\{\bigl\Vert V_{n}^{\xi}\bigr\Vert_{I}\leq z\mid\mathsf{D}_{n}\bigr\} =\mathrm{P}\bigl\{ \bigl\Vert T_{n}^{G}\bigr\Vert_{I}\leq z\bigr\}$ for all $z\in\mathbb{R}$ almost surely, and such that 
\begin{equation}
\label{DifTMB}
\bigl|\bigl\Vert\mathsf{G}_{n}^{\xi}\bigr\Vert_{\mathcal{F}_{i,n}}-V_{n}^{\xi}\bigr|=\mathcal{O}_{P}\biggl\{ \frac{(\log n)^{2+1/q}}{n^{1/2-1/q}\sqrt{\Delta h_n^{-1}}}+\frac{(\log n)^{7/4+1/q}}{(n\Delta h_n^{-1})^{1/4}}\biggr\} =\mathcal{O}_{P}\biggl\{ \frac{(\log n)^{7/4+1/q}}{(n\Delta h_n^{-1})^{1/4}}\biggr\}
\end{equation}
This in turn implies that there exists a sequence of constants $\varepsilon_{n}\to0$ such that 
\begin{equation*}
\mathrm{P}\biggl\{ \bigl|\bigl\Vert \mathsf{G}_{n}^{\xi}\bigr\Vert _{\mathcal{F}_{i,n}}-V_{n}^{\xi}\bigr|\geq\varepsilon_{n}\frac{(\log n)^{7/4+1/q}}{(n\Delta h_n)^{1/4}}\mid\mathsf{D}_{n}\biggr\}\stackrel{p}{\to}0.
\end{equation*} 
The condition \ref{ass5} of Assumption~\ref{assumption1} guarantees that the expression \eqref{DifTMB} converges to \(0\) and with speed faster than one of the expression \eqref{DifT}.
Since $\bigl\Vert \mathbb{G}_{n}^{\xi}\bigr\Vert _{\mathcal{F}_{i,n}}=\bigl\Vert T_{n}^{MB}\bigr\Vert _{I}$, we get together with the bound $\mathsf{E}\bigl[\bigl\Vert T_{n}^{G}\bigr\Vert_{I}\bigr]\lesssim(\log n)^{1/2}$ and the anti-concentration inequality (\ref{anticon}), 
\begin{equation*}
\begin{split}
\mathrm{P}\bigl\{\bigl\Vert T_{n}^{MB}\bigr\Vert_{I}\leq z\mid\mathsf{D}_{n}\bigr\}  & \leq\mathrm{P}\biggl\{ \bigl\Vert V_{n}^{\xi}\bigr\Vert_{I}\leq z+\varepsilon_{n}h_n^{1/2}\log h_n^{-1 }\mid\mathsf{D}_{n}\biggr\} +o_{p}(1)\\
 & =\mathrm{P}\biggl\{\bigl\Vert T_{n}^{G}\bigr\Vert_{I}\leq z+\varepsilon_{n}h_n^{1/2}\log h_n^{-1 }\biggr\} +o_{p}(1)\leq\\
 & \leq\mathrm{P}\bigl\{\bigl\Vert T_{n}^{G}\bigr\Vert_{I}\leq z\bigr\} +o_{p}(1)
\end{split}
\end{equation*}
For the same reason, we conclude that
\begin{equation*}
\mathrm{P}\bigl\{\bigl\Vert T_{n}^{MB}\bigr\Vert_{I}\leq z\mid\mathsf{D}_{n}\bigr\} \geq\mathrm{P}\bigl\{\bigl\Vert T_{n}^{G}\bigr\Vert_{I}\leq z\bigr\} -o_{p}(1).
\end{equation*} 
This argument shows together with (\ref{TMB}) that 
\begin{equation}
\sup_{z\in\mathsf{R}}\bigl|\mathrm{P}\bigl\{ \bigl\Vert \widehat{T}_{n}^{MB}(\cdot)\bigr\Vert _{I}\leq z\mid\mathsf{D}_{n}\bigr\}-\mathrm{P}\bigl\{\bigl\Vert T_{n}^{G}\bigr\Vert_{I}\leq z\bigr\} \bigr|\stackrel{p}{\to}0.
\label{supTmb}
\end{equation}
To conclude the proof, it remains to show that
\begin{equation}
\mathrm{P}\bigl\{\nu(x)\in\widehat{\mathcal{C}}_{1-\tau}^{MB}(x)\quad\forall\, x\in I\bigr\} \to1-\tau.
\end{equation}
Let us recall that it follows from Theorem \ref{thm1} together with the bound
 $\mathsf{E}\bigl[\bigl\Vert T_{n}^{G}\bigr\Vert_{I}\bigr]\lesssim(\log n)^{1/2}$ that
\begin{equation*}
\rho(x)\in\widehat{\mathcal{C}}_{1-\tau}^{MB}(x)\quad\forall x\,\in I, \text{ if and only if } \biggl \Vert \frac{\sqrt{n}\Delta (\widehat{\rho}_{n}(\cdot)-\rho(\cdot))}{\widehat{s}_{n}(\cdot)}\biggr\Vert _{I}\leq\widehat{c}_{n}^{MB}(1-\tau)
\end{equation*}
and we have $\bigl\Vert T_{n}\bigr\Vert_{I}=\mathcal{O}_{p}\bigl\{ (\log n)^{1/2}\bigr\}.$
Let us remark that
\begin{equation*}
\begin{split}
\frac{\sqrt{n}\Delta(\widehat{\rho}_{n}(x)-\rho(x))}{\widehat{s}_{n}(x)} & =\frac{s(x)}{\widehat{s}_{n}(x)}\frac{\sqrt{n}\Delta(\widehat{\rho}_{n}(x)-\rho(x))}{s(x)}\\
 & =\bigl(1+o_{p}\bigl\{ n^{-1/2}\log h_n^{-1}\bigr\} \bigr)\biggl[T_{n}(x)+o_{p}\bigl( h_n^{1/2}\log h_n^{-1}\bigr) \biggr]\\
 & =T_{n}(x)+o_{p}\bigl( h_n^{1/2}\log h_n^{-1}\bigr). 
\end{split}
\end{equation*}
Now if we recall the conclusion of Theorem \ref{thm1} and the anti-concentration inequality (\ref{anticon}), we get 
\begin{equation}
\sup_{z\in\mathbb{R}}\biggl|\mathrm{P}\biggl\{\biggl\Vert \frac{\sqrt{n}\Delta\bigl(\widehat{\rho}_{n}(\cdot)-\rho(\cdot)\bigr)}{\widehat{s}_{n}(\cdot)}\biggr\Vert _{I}\leq z\biggr\} -\mathrm{P}\bigl\{\bigl\Vert T_{n}^{G}\bigr\Vert_{I}\leq z\bigr\} \biggr|\to0.
\label{supNull}
\end{equation}
Note that due to (\ref{supTmb}) together with argument similar to Step 3 in the proof of Theorem \ref{theorem2} in \cite{kato2018uniform}, we can find a sequence of constants $\varepsilon_{n}'\to0$ such that
\begin{equation}
c_{n}^{G}(1-\tau-\varepsilon_{n}')\leq\widehat{c}_{n}^{MB}(1-\tau)\leq c_{n}^{G}(1-\tau+\varepsilon_{n}')
\label{cInt}
\end{equation}
with probability approaching one. 
This implies that
\begin{equation*}
\begin{split} 
& \mathrm{P}\biggl\{\biggl\Vert \frac{\sqrt{n}\Delta(\widehat{\rho}_{n}(\cdot)-\rho(\cdot))}{\widehat{s}_{n}(\cdot)}\biggr\Vert _{I}\leq\widehat{c}_{n}^{MB}(1-\tau)\biggr\} \\
 &\hspace{2cm} \stackrel{(\ref{cInt})}{\leq}\mathrm{P}\biggl\{ \biggl\Vert \frac{\sqrt{n}\Delta (\widehat{\rho}_{n}(\cdot)-\rho(\cdot))}{\widehat{s}_{n}(\cdot)}\biggr\Vert _{I}\leq c_{n}^{G}(1-\tau+\varepsilon_{n}')\biggr\} +o(1)\\
 &\hspace{2cm} \stackrel{(\ref{supNull})}{=}\mathrm{P}\bigl\{ \bigl\Vert T_{n}^{G}\bigr\Vert_{I}\leq c_{n}^{G}(1-\tau+\varepsilon_{n}')\bigr\} +o(1)=1-\tau+o(1)
\end{split}
\end{equation*} 
For the same reason, we have upper bound for the probability,  which has the form $1- \tau-o(1)$.
Due the Borell-Sudakov-Tsirelson inequality (see Lemma A.2.2 in  \cite{van1996weak} for more details) we have 
\begin{equation*}
c_{n}^{G}(1-\tau+\varepsilon_{n}')\lesssim\mathsf{E}\bigl[\bigl\Vert T_{n}^{G}\bigr\Vert_{I}\bigr]+\sqrt{1+\log(1/(\tau-\varepsilon'))}\lesssim(\log n)^{1/2}.
\end{equation*}
If we combine this with \eqref{cInt}, we get $\widehat{c}_{n}^{MB}(1-\tau)=O_{P}\bigl(\sqrt{\log n}\bigr)$ with the supremum width of the confidence band $\mathcal{C}_{1-\tau}^{MB}$ bounded as 
\begin{equation*}
\begin{split}
2\,\underset{x\in I}{\sup}\,\frac{\widehat{s}_{n}(x)}{\sqrt{n}\Delta}\,\widehat{c}_{n}^{MB}(1-\tau)&\lesssim\bigl( 1+o_{P}(1)\bigr) \frac{\sup_{x\in I}s(x)}{\sqrt{n}\Delta}\,\widehat{c}_{n}^{MB}(1-\tau)\\
&=\mathcal{O}_{p}\bigl((n\Delta h_n^3)^{-1/2} \sqrt{\log n}\bigr)
\end{split}
\end{equation*} 
 This observation completes the proof of Theorem \ref{theorem2} for the multiplier bootstrap case.

\bibliographystyle{plain} 
\bibliography{bibliography-final-1}

\end{document}